\documentclass[11pt]{article}
\usepackage[verbose, a4paper, hcentering, width=6.5in]{geometry}	
\usepackage{amsthm, mathtools, amssymb, nicefrac}
\mathtoolsset{showonlyrefs}

\usepackage{algorithm}
\usepackage{algorithmic}
\usepackage{booktabs}
\usepackage{graphicx}    
\usepackage{subcaption}  
\usepackage{pdfpages}
\usepackage{todonotes}
\usepackage{float} 
\usepackage{dsfont}
\usepackage{bm}
\usepackage[numbers]{natbib}	
\usepackage[bookmarks=true,
    pdfpagemode= UseNone,
    backref= page, 
    pdftitle={Maximum Mean Discrepancy},
    pdfauthor={Mehraban},
    colorlinks=true, citecolor=gray, urlcolor=darkgray, linkcolor=darkgray,
    pdfstartview= FitH]{hyperref} 
\usepackage{orcidlink}
\numberwithin{equation}{section}   
\theoremstyle{plain}
	\newtheorem{theorem}             {Theorem} 
	\newtheorem{corollary}  [theorem]{Corollary}
	
	\newtheorem{proposition}[theorem]{Proposition}
\theoremstyle{definition}
	
\theoremstyle{remark}
	\newtheorem{remark}     [theorem]{Remark}

\DeclareMathOperator{\E}{\mathds{E}}	        
\DeclareMathOperator{\one}{\mathds{1}}	        

\DeclareMathOperator{\MMD}{MMD}	
\DeclareMathOperator*{\argmin}{arg\,min}

\usepackage{doi}
\usepackage[all,defaultlines=2]{nowidow}
\title{Quantization Of Probability Measures In Maximum~Mean~Discrepancy Distance}
\author{
    Zahra Mehraban\footnote{Faculty of Mathematics, University of Technology, D-09111 Chemnitz, Germany}\, %
        \thanks{\href{mailto:zahra.mehraban@s2022.tu-chemnitz.de}{zahra.mehraban@s2022.tu-chemnitz.de}}\and
    Alois Pichler\footnotemark[1]\,~%
        \thanks{\orcidlink{0000-0001-8876-2429}~\href{https://orcid.org/0000-0001-8876-2429}{orcid.org/0000-0001-8876-2429}
}}

\begin{document}
    \maketitle
    \begin{abstract}
      Accurate approximation of probability measures is essential in numerical applications. This paper explores the quantization of probability measures using the \emph{maximum mean discrepancy} (MMD) distance as a guiding metric. We first investigate optimal approximations by determining the best weights, followed by addressing the problem of optimal facility locations.

      To facilitate efficient computation, we reformulate the nonlinear objective as expectations over a product space, enabling the use of stochastic approximation methods. For the Gaussian kernel, we derive closed-form expressions to develop a deterministic optimization approach. By integrating stochastic approximation with deterministic techniques, our framework achieves precise and efficient quantization of continuous distributions, with significant implications for machine learning and signal processing applications.
    \end{abstract}
    \noindent\textbf{Classification: } 62F12, 62G20, 68T07, 90C31

\section{Introduction}
Accurate approximations of probability measures are essential when working with measures numerically. A commonly used approach is the approximation of probability measures by a simple measure, often referred to as \emph{quantization}.

The accuracy of the approximation has been measured differently in the literature, with significant recent developments relying on the Wasserstein distance.
Foundational work on approximation quality in Wasserstein distance is provided in \citet{GrafLuschgy}, including the approximation quality in Wasserstein distance.
Additional theoretical advancements, along with applications, are extensively covered in the works of Pagès, such as \citet{PagesBallyPrintems, PagesSpaceQuantization}, and \citet{PagesFunctionalQuantization}.

This paper investigates an alternative approach: approximating probability measure using the maximum mean discrepancy (MMD).
In this framework, the probability measure is embedded into a suitable Hilbert space of functions, where a norm is employed for comparison.
Specifically, MMD leverages a reproducing kernel Hilbert space (RKHS), enabling a computationally efficient evaluation for discrepancies between distributions. 

Reliable approximations of probability measures are crucial across numerous applications.
In optimization under uncertainty, these results are indispensable to guarantee continuity of the stochastic program.
Other fields that benefit include statistical learning, information theory, mathematical finance, data compression, signal processing, and machine learning, where infinite-dimensional distributions are often impractical to handle.
The primary challenge is identifying a finite set of points that best represent the original distribution while minimizing the approximation error.
For a deeper exploration of related topics, see \citet{Gersho2012, Gray1998, Widrow2008, Hubara2018, Liang2021} and \citet{Rajasekar2020}.

Recently, MMD has gained prominence as a powerful tool for comparing probability distributions. Unlike the Wasserstein distance, which is computationally expensive due to the optimization problem over transport plans, MMD offers a more efficient alternative. By utilizing kernel functions to assess discrepancies in a high-dimensional feature spaces, the MMD is particularly well-suited for large-scale statistical learning tasks where minimizing computation time is essential.

Quantization in Wasserstein distance suffers from the curse of dimensionality, as discussed in \citet{Bolley}.
For this reason, regularization techniques have been developed and investigated as well, cf.\ \citet{LakshmananPichlerQuantization, LakshmananPichler}, as well as accelerated computational strategies, cf.\ \citet{LakshmananPichlerPotts}. 
More recently, MMD has been applied to quantization problems.
For instance, \citet{Feng2021} introduce Deep Weibull Hashing (DWH), combining MMD quantization with a Weibull constraint to enhance nearest neighbor search.
Experimental results demonstrate that DWH outperforms existing methods by better preserving pairwise similarities. Similarly, \citet{Teymur2021} developed sequential algorithms for MMD minimization, including a new non-myopic approach and a mini-batch variant to improve efficiency.
Other notable contributions include \citet{Pronzato2022}, who evaluate various MMD-based quantization algorithms --~such as kernel herding, greedy minimization, and sequential Bayesian quadrature~-- and \citet{Xu2022}, who introduces adaptive kernel methods for particle-based approximation, achieving improvements in robustness and errors reduction.
Further applications are found in \citet{GuoXu2017} and \citet{XuQuantification2024}.

In this paper, we progress optimizing quantization by minimizing MMD through a Stochastic Gradient Descent (SGD) algorithm.

\paragraph{Outline of the paper. }
Section~\ref{sec:Preliminaries} recalls necessary components from reproducing kernel Hilbert spaces.
Section~\ref{sec:2} lays out the theoretical foundations of MMD and introduces our proposed SGD algorithm. The next section (Section~\ref{sec:3}) presents optimization results across various distributions, highlighting the effectiveness of our approach.
Finally, numerical results are showcased in Section~\ref{sec:Results}.

\section{Mathematical setting}\label{sec:Preliminaries}
To formally define the optimization problem consider a reproducing kernel Hilbert space ${(\mathcal H_k,\,\|\cdot\|_k)}$ of functions on~$\mathcal X$ with kernel~$k$. The inner product that ensures the reproducing property for this space is given by
\[	\langle k(x,\cdot)|\ k(y,\cdot)\rangle_k= k(x,y) 
\]
(cf.\ \citet{Berlinet2004} for a discussion of RKHS spaces).
Given a general measure~$\mu$ on $\mathcal X$, the \emph{kernel embedding function} is defined as
\begin{equation}\label{eq:1}
  \mu_k(\cdot)\coloneqq \int_{\mathcal X} k(\cdot,\xi)\,\mu(\mathrm{d}\xi).
\end{equation}
The mapping $\mu\mapsto \mu_k$ embeds measures~$\mu$ into $\mathcal H_k$, and the \emph{maximum mean discrepancy} distance between two measures~$\mu$ and~$\nu$ is given by
\[ \MMD_k(\mu,\nu)\coloneqq \|\mu_k-\nu_k\|_k.\]
Unlike the Wasserstein distance, the MMD does not require the measures~$\mu$ and~$\nu$ to have the same mass. In other words, they do not necessarily need to satisfy the condition $\mu(\mathcal X)= \nu(\mathcal X)$ for the distance to be well-defined.
Explicitly, the maximum mean discrepancy distance of two measures on $\mathcal X$ is
\begin{align}
  \MMD_k(\mu,\nu)^2 &= \iint_{\mathcal X\times\mathcal X} k(\xi,\xi^\prime)\mu(\mathrm{d}\xi)\mu(\mathrm{d}\xi^\prime)&  \label{eq:2}\\
       & \quad - 2\iint_{\mathcal X\times\mathcal X}k(\xi,\xi^\prime)\mu(\mathrm{d}\xi)\nu(\mathrm{d}\xi^\prime)+ \iint_{\mathcal X\times\mathcal X}k(\xi,\xi^\prime)\nu(\mathrm{d}\xi)\nu(\mathrm{d}\xi^\prime). \nonumber
\end{align}
A recent review on the subject can found in \citet{MuandetScholkopf2017}.

This paper focuses on the quantization of measures, specifically the approximation of a measure~$\mu$ by a simple, discrete measure of the form
\[ \sum_{i=1}^n \mu_i\cdot\delta_{x_i},\]
where the goal is to achieve the best possible approximation.
Our focus is on probability measures.
Given a probability measure~$P$, the quantization problem involves finding a discrete \emph{probability} measure of the form
\begin{equation}
	P^n\coloneqq \sum_{i=1}^n p_i\cdot \delta_{x_i} \label{eq:3}
\end{equation}
that approximates~$P$ by as accurately as possible.
In~\eqref{eq:3}, the weights $p_1,\ p_2,\dots,\ p_n$ are non-negative and $\sum_{i=1}^n p_i=1$, and~$\delta_{x}$ is the Dirac measure located at~$x$.%
\footnote{The Dirac measure is $\delta_a(A)\coloneqq \one_A(a)= \begin{cases}
 1 &	\text{if }a\in A,	\\
 0 &	\text{else}.
\end{cases}$}
The support points $x_1,\dots, x_n$ of the finite measure~$P^n$ are called \emph{quantizers}, the embedding of the approximating measure is the function
\[ P^n_k(\cdot)= \sum_{i=1}^n p_i\cdot k(x_i,\cdot)\in \mathcal H_k,\]
cf.~\eqref{eq:1}.
With that, the quantization problem is 
\begin{subequations}\label{eq:4}
  \begin{align} 
      \text{minimize} \quad & \Vert P_k - P_k^n \Vert_k \label{P1}\\
      \text{in}\quad & \bm{p}= [p_1,p_2,\dots,p_n] \text{ and } \bm{x}= [x_1,\dots,x_n],\label{P5}\\
      \text{subject to} \quad & \sum_{i=1}^n p_i = 1, \label{P2}\\
      & p_i\ge 0, \text{ and} \label{P3}\\
      & x_i\in \mathcal X \text{ for } i=1,\dots,n, \label{P4}
  \end{align}
\end{subequations}
where the expanded objective~\eqref{P1} is
\begin{align}
   \Vert P_k - P_k^n \Vert_k^2 & =  \iint_{\mathcal X\times \mathcal X} k(\xi,\xi')\,\mu(\mathrm{d}\xi)\mu(\mathrm{d}\xi')\nonumber\\
  &\qquad - 2 \sum_{i=1}^n p_i\int_\mathcal X k(\xi,x_i)\,\mu(\mathrm{d}\xi) + \sum_{i,j=1}^n k(x_i, x_j)\,p_i\,p_j \label{DMMD}
\end{align}
by~\eqref{eq:2}.

\section{The quantization weights}\label{sec:2}
We shall address the problem of optimal quantization~\eqref{P1}--\eqref{P4} in consecutive steps: first, we describe the optimal weights in the unconstrained problem~\eqref{P1} and the problem~\eqref{P1}--\eqref{P3}; next, we consider the problem of finding the optimal quantizers~\eqref{P4} separately. Given these results, we shall finally address the problem of locating the optimal quantization points.

\subsection{Approximations of general measures}
Consider the following general version of the problem~\eqref{P1} first to identify the optimal weights for fixed $x_i\in \mathcal X$ ($x_i\not= x_j$ for $i\not=j$):
\begin{align}
    \text{minimize }\quad & \|\mu_k- \mu_k^n\|_k \label{eq:5},\\
    \text{where } \quad& \mu^n= \sum_{i=1}^n\mu_i\,\delta_{x_i} 
  \text{ and }   \mu_1,\dots, \mu_n\in \mathbb R.\nonumber
\end{align}
Here, $\mu$ is a general measure, not necessarily a probability measure.
The optimization problem aims at identifying the optimal weights $\mu_1,\dots, \mu_n$ of the approximating discrete measure $\mu^n= \sum_{i=1}^n \mu_i\,\delta_{x_i}$.
\begin{proposition}\label{prop:180}
  The optimal weights $\bm \mu= [\mu_1,\dots,\mu_n]$ in~\eqref{eq:5} are given by
  \begin{equation}\label{eq:6}
    \bm\mu= \bm{K}^{-1}\,\bm{m},
  \end{equation}
  where $\bm{K}$ is the matrix with entries
  \[ \bm{K}_{ij}= k(x_i,x_j),\qquad i,j=1,\dots,n,\]
  and the entries of the vector~$\bm{m}$ are evaluations of the embedding of the initial function at the support points, 
  \[
    m_i\coloneqq  \mu_k(x_i)= \int_{\mathcal X}k(x_i,\xi)\,\mu(d\xi)\]
  for $i=1,\dots,n$.
\end{proposition}
\begin{proof}
  The necessary conditions of the optimization problem with objective~\eqref{eq:5} are
 \[
    0= \frac{\partial L}{\partial \mu_i}= -2\int_{\mathcal X}k(\xi,x_i)\mu(\mathrm{d}\xi)+ 2\sum_{j=1}^n \mu_j\, k(x_i,x_j),
  \]
  where~\eqref{DMMD} expands the Lagrangian
 \[ L(\bm \mu)= L(\mu_1, \ldots, \mu_n) \coloneqq \Vert \mu_k - \mu_k^n \Vert_k^2
 \]
  explicitly. As the matrix $\bm{K}$ is symmetric, the assertion~\eqref{eq:6} is immediate.
  Finally, note that
  \[	\mu_k(x_i)= \int_{\mathcal X}k(x_i,\xi)\mu(\mathrm d\xi)\]
  by the definition in~\eqref{eq:1}, and hence the assertion.
\end{proof}
\begin{corollary}\label{cor:2}
  For fixed support points $\bm x= [x_1,\dots,x_n]$, the optimal maximum mean discrepancy distance is given by
\begin{equation}\label{eq:7}
  \MMD{(\mu,\mu^n)}^2 = \iint_{\mathcal X\times\mathcal X}k(\xi,\xi^\prime)\mu(\mathrm d\xi)\mu(\mathrm d\xi^\prime)- \bm{m}^\top\bm{K}^{-1}\,\bm{m}.
\end{equation}
\end{corollary}
\begin{proof}
  With $\bm\mu=\bm{K}^{-1}\bm{m}$ from Proposition~\ref{prop:180}, the objective~\eqref{DMMD} is
  \begin{align*}
    \MMD{(\mu,\mu^n)}^2&= \|\mu_k- \mu_k^n\|_k^2\\
              &= \iint_{\mathcal X\times\mathcal X}k(\xi,\xi^\prime)\mu(\mathrm d\xi)\mu(\mathrm d\xi^\prime)- 2\bm{m}^\top\bm{K}^{-1}\,\bm{m}+ \bm{m}^\top \bm{K}^{-1}\bm{K}\bm{K}^{-1}\,\bm{m}\\
              &= \iint_{\mathcal X\times\mathcal X}k(\xi,\xi^\prime)\mu(\mathrm d\xi)\mu(\mathrm d\xi^\prime)- \bm{m}^\top\bm{K}^{-1}\,\bm{m},
  \end{align*}
  thus the assertion.
\end{proof}
\begin{remark}
  
  For~$\xi\in \mathcal X$ fixed, consider the function $f(\cdot)\coloneqq \sum_{i=1}^n w_i\,k(x_i,\cdot)$ with weights $\bm w=\bm K^{-1}k(\xi,\bm x)$, where $k(\xi,\bm x)\coloneqq \bigl(k(\xi,x_1),\dots,k(\xi,x_n)\bigr)$.
  By construction, it holds that $f(x_i)= k(x_i,\xi)$, that is, the function $f(\cdot)$ interpolates $k(\cdot,\xi)$ at $x_1\dots,x_n$.
  By the representer theorem (cf.\ \citet{Schoelkopf2001}), the function~$f(\cdot)$ has the smallest norm ($\|\cdot\|_k$) among all functions interpolating the function $k(\cdot,\xi)$ in the points $x_i$, $i=1,\dots,n$.
  That is,
  \[
    \sum_{i,j=1}^n k(x_i,\xi)\bm K_{ij}^{-1}k(x_j,\xi)= \|f\|_k^2\le \|k(\cdot,\xi)\|_k^2= k(\xi,\xi).\]
  This inequality corresponds to $\MMD(\delta_\xi,\delta_\xi^n)\ge 0$ for the specific measure $\mu=\delta_\xi$ in Corollary~\ref{cor:2}.
\end{remark}

\subsection{Approximation of probability measures}
Even for an initial probability measure $\mu= P$ in~\eqref{eq:5}, the optimal approximating measure $\sum_{i=1}^n \mu_i\cdot\delta_{x_i}$, given in Proposition~\ref{prop:180} above, is --~in general~-- not a probability measure: the weights $\mu_i$ in~\eqref{eq:6} are not necessarily non-negative, and the weights do not sum to one.

We have the following result on the optimal weights of the probability measure approximating a probability measure~$P$.
\begin{theorem}\label{TP}
For fixed support points $\bm x= [x_1,\dots,x_n]$, the probability measure with the smallest maximum mean discrepancy distance in~\eqref{P1}--\eqref{P2} has weights 
\begin{align}\label{eq:8}
   \bm{p}= \left(  \bm{K}^{-1} -\frac{\bm{K}^{-1} \bm{1}  \bm{1}^\top \bm{K}^{-1}}{\bm{1}^\top \bm{K}^{-1} \bm{1}} \right) \bm{m}+\frac{\bm{K}^{-1} \bm{1}}{\bm{1}^\top \bm{K}^{-1} \bm{1}},
\end{align}
where $\bm{m}$ is the vector with entries
\begin{equation}\label{eq:9}
  m_i= \E k(\cdot,x_i)= \int_{\mathcal X}k(x_i,\xi)\,P(\mathrm d\xi)= P_k(x_i)
\end{equation}
and $P_k(\cdot)=\int_{\mathcal X}k(\cdot,\xi)\,P(\mathrm d\xi)\in \mathcal H_k$ is the embedding of~$P$, cf.~\eqref{eq:1}, and $\bm 1\coloneqq [1,\dots,1]$.
\end{theorem}
\begin{proof}
The Lagrangian function of the optimization problem~\eqref{eq:4} with constraints~\eqref{P2} is
\begin{align*}
  L(p_1, \ldots, p_n; \lambda) &\coloneqq \Vert P_k - P_k^n \Vert_k^2 + \lambda \left(\sum_{i=1}^n p_i - 1\right) \\
                  & = \iint_{\mathcal X\times \mathcal X} k(\xi,\xi') \, P(\mathrm{d}\xi) \, P(\mathrm{d}\xi')- 2 \sum_{i=1}^n p_i\int_{\mathcal X} k(\xi,x_i) \, P(\mathrm{d}\xi) \\
                 & \qquad  + \sum_{i,j=1}^n k(x_i, x_j) \, p_i \, p_j +\lambda \left(\sum_{i=1}^n p_i - 1\right).
\end{align*}
To extract the first order conditions, differentiate the Lagrangian function with respect to $p_\ell$ ($\ell= 1,\ldots,n$) and the multiplier $\lambda$, and then equate the derivatives to zero.
This gives the necessary conditions
    \begin{align}\label{eq:10}
        0=\frac{\partial L}{\partial p_\ell} &= -2\int_{\mathcal X} k(\xi,x_\ell)\,P(\mathrm{d}\xi) + 2 \sum_{i=1}^n k(x_i,x_\ell) p_i + \lambda, \quad \ell = 1,\ldots,n,
    \end{align}
    and the constraint
    \begin{align}\label{eq:11}
        0=\frac{\partial L}{\partial \lambda} &= \sum_{i=1}^n p_i - 1.
    \end{align}
Multiplying both sides of the first equation by $p_i$ and summing with respect to the index~$i$, we obtain
\begin{align*}
    -2\sum_{i=1}^n  p_i \int_{\mathcal X} k(\xi,x_i) \, P(\mathrm{d}\xi) &+2 \sum_{i,j=1}^n k(x_i,x_j) p_i p_j + \lambda \sum_{i=1}^n  p_i = 0.
\end{align*}
Since $\sum_{i=1}^n p_i = 1$ from~\eqref{eq:11}, the last term simplifies to $\lambda$.
Therefore, the system of equations becomes
\begin{align*}
 2\sum_{i=1}^n  p_i \int_{\mathcal X} k(\xi,x_i)\, P(\mathrm{d}\xi) &-2 \sum_{i,j=1}^n k(x_i,x_j) p_i p_j- \lambda= 0.
\end{align*}
As $\int_{\mathcal X} k(\xi,x_i)\, P(\mathrm{d}\xi)= \E k(\cdot,x_i)$ for $i=1,.., n$, we have
\begin{equation}
\lambda  = 2\sum_{i=1}^n p_i \E k(\cdot,x_i) -2 \sum_{i,j=1}^n k(x_i,x_j) p_i p_j,
\end{equation}
that is,
\begin{align}\label{eq:12}
\lambda = 2\bm{p}^\top \bm{m} - 2 \bm{p}^\top \bm{K} \bm{p},
\end{align}
where $\bm{p}= {[p_1,p_2,\ldots, p_n]}^\top$ is the vector of probabilities,
$\bm{m}= {[m_1,m_2,\ldots, m_n]}^\top$ is the vector of expected values of the kernel function, and~$\bm{K}$ is the kernel matrix.


Given the value obtained for $\lambda$ and the specified constraint  $\sum_{i=1}^n p_i= 1$, we conclude from~\eqref{eq:10} that
\begin{align}\label{KPM}
  \bm{K}\bm{p} + \bm{1} \left( \lambda/2\right) = \bm{m},
\end{align}
where $\bm{1}$ is an $n$-dimensional column vector of ones. 
Multiplying equation~\eqref{KPM} by $\bm{p}^\top$ yields $\lambda$, and since $\bm{p}^\top\bm{1}=1$, we find the equations
\begin{align}\label{K'P'M'}
  \bm{K}' \bm{p}'=\bm{m}',
\end{align}
where
\begin{align}\label{eq:13}
\bm{K}' = \begin{pmatrix}
\bm{K} & \bm{1} \\
\bm{1}^\top & 0
\end{pmatrix},\,\,\,
\bm{p}^\prime& = \begin{pmatrix}
\bm{p} \\
\lambda/2
\end{pmatrix}\,\,\,\text{and}\,\,\,
\bm{m}' = \begin{pmatrix}
\bm{m} \\
1
\end{pmatrix}.
\end{align}
The Woodbury matrix identity (or Sherman--Morrison formula) can be applied to compute the inverse 
\begin{align}
\mathbf{K'}^{-1} 
&=\label{eq:14}
\begin{pmatrix}
  \bm{K}^{-1}-\bm{K}^{-1}\bm{1}{\left(\bm{1}^\top\bm{K}^{-1}\bm{1}\right)}^{-1}\bm{1}^\top\bm{K}^{-1}  & \bm{K}^{-1}\bm{1}\left(\bm{1}^\top\bm{K}^{-1} \bm{1} \right)^{-1} \\
  \left( \bm{1}^\top\bm{K}^{-1}\bm{1}\right)^{-1} \bm{1}^\top\bm{K}^{-1} & -\left( \bm{1}^\top\bm{K}^{-1}\bm{1}\right)^{-1} .
\end{pmatrix},
\end{align}
From this the optimal solution for~$\bm{p}$ in~\eqref{eq:8} and
\[  \lambda/2= \frac{\bm{1}^\top \bm{K}^{-1} \bm{m}-1}{\bm{1}^\top \bm{K}^{-1} \bm{1}}\] derive.
Instead of verifying~\eqref{eq:14}, we validate the result~\eqref{eq:8} for the optimal weights derived in Theorem~\ref{TP}.
We may verify~\eqref{K'P'M'} with respect to the expression for~$\bm{p}$ from~\eqref{eq:8} and by substituting the value of $\lambda$ from~\eqref{eq:12}.
Specifically, we need to demonstrate~\eqref{eq:13}, that is,
\begin{align*}
\begin{pmatrix}
\bm{K}\bm{p}+ \bm{1}\bm{p}^\top\bm{m}-\bm{1}\bm{p}^\top\bm{K}\bm{p} \\
\bm{1}^\top\bm{p} 
\end{pmatrix}=\begin{pmatrix}
\bm{m} \\
1 
\end{pmatrix}.
\end{align*}

To verify the expression, we calculate each term on the left-hand side. 
First, note that
\begin{align*}
\bm{K}\bm{p} = \bm{m} - \frac{\bm{1} \, \bm{1}^\top \bm{K}^{-1} \bm{m}}{\bm{1}^\top \bm{K}^{-1} \bm{1}} + \frac{\bm{1}}{\bm{1}^\top \bm{K}^{-1} \bm{1}},
\end{align*}
where we assume~\eqref{eq:8}.
Similarly, we have
\begin{align*}
\bm{1} \bm{p}^\top \bm{m} = \bm{1} \left( \bm{m}^\top \bm{K}^{-1} \bm{m} - \frac{\bm{m}^\top \bm{K}^{-1} \bm{1} \, \bm{1}^\top \bm{K}^{-1} \bm{m}}{\bm{1}^\top \bm{K}^{-1} \bm{1}} + \frac{\bm{1}^\top \bm{K}^{-1} \bm{m}}{\bm{1}^\top \bm{K}^{-1} \bm{1}} \right)
\end{align*}
and
\begin{align*}
-\bm{1} \bm{p}^\top \bm{K} \bm{p} = -\bm{1} \left( \bm{m}^\top \bm{K}^{-1} \bm{m} - \frac{\bm{m}^\top \bm{K}^{-1} \bm{1} \, \bm{1}^\top \bm{K}^{-1} \bm{m}}{\bm{1}^\top \bm{K}^{-1} \bm{1}} + \frac{\bm{1}^\top \bm{K}^{-1} \bm{m}}{\bm{1}^\top \bm{K}^{-1} \bm{1}} \right).
\end{align*}

By combining these expressions, we observe that the terms involving $\bm{1}$ and $\bm{m}$ cancel appropriately, leaving
\begin{align*}
\bm{K}\bm{p} + \bm{1} \bm{p}^\top \bm{m} - \bm{1} \bm{p}^\top \bm{K} \bm{p} = \bm{m}.
\end{align*}

Thus, the left-hand side reduces to $\begin{pmatrix} \bm{m} \\ 1 \end{pmatrix}$, as required.
This confirms that the optimal weights provide a precise representation of the continuous  distribution.
\end{proof}

The following corollary characterizes the objective by the optimal probabilities identified above.
\begin{corollary}\label{cor:412}
  For fixed support points $\bm x= [x_1,\dots,x_n]$, the objective with optimal weights is
  \begin{align*}
    \MMD(P,P^n)^2&= \iint_{\mathcal X\times\mathcal X}k(\xi,\xi^\prime)P(\mathrm d\xi)P(\mathrm d\xi^\prime)- \bm m^\top \bm K^{-1}\bm m+ \frac{(\bm 1^\top\bm K^{-1}\bm m-1)^2}{\bm 1^\top\bm K^{-1}\bm 1}.
\end{align*}
\end{corollary}
\begin{proof}
  Involving the optimal vector $\bm p$ in~\eqref{eq:8} in Theorem~\ref{TP} from~\eqref{DMMD} 
  \begin{align*}
  {\MMD(P,P^n)}^2&= \|P_k- P_k^n\|_k^2 \\
      &= \iint_{\mathcal X\times\mathcal X}k(\xi,\xi^\prime)P(\mathrm d\xi)P(\mathrm d\xi) \\
      &\quad -2\left(\bm m^\top\bm K^{-1}-\frac{\bm m^\top\bm K^{-1}\bm 1\bm 1^\top\bm K^{-1}}{\bm 1^\top\bm K^{-1}\bm 1}+\frac{\bm 1^\top\bm K^{-1}}{\bm 1^\top\bm K^{-1}\bm 1}\right)\bm m\\
      &\quad +\left(\bm m^\top\bm K^{-1}-\frac{\bm m^\top\bm K^{-1}\bm 1\bm 1^\top\bm K^{-1}}{\bm 1^\top\bm K^{-1}\bm 1}+\frac{\bm 1^\top\bm K^{-1}}{\bm 1^\top\bm K^{-1}\bm 1}\right)\\
    & \qquad \quad \cdot \bm K\cdot \left(\bm K^{-1}\bm m-\frac{\bm K^{-1}\bm 1\bm 1^\top\bm K^{-1}\bm m}{\bm 1^\top\bm K^{-1}\bm 1}+\frac{\bm K^{-1}\bm 1}{\bm 1^\top\bm K^{-1}\bm 1}\right).
  \end{align*}
  Collecting terms, it follows that
    \begin{align*}
     \MMD(P,P^n)^2
      &= \iint_{\mathcal X\times\mathcal X}k(\xi,\xi^\prime)P(\mathrm d\xi)P(\mathrm d\xi) \\
      &\qquad -(2-1)\bm m^\top \bm K^{-1}\bm m+ (2-1-1)\frac{\bm m^\top \bm K^{-1}\bm 1\bm 1^\top\bm K^{-1}\bm m}{\bm 1^\top\bm K^{-1}\bm 1}- (2-1-1)\frac{\bm 1^\top \bm K^{-1}\bm m}{\bm 1^\top\bm K\bm 1}\\
      &\qquad +\frac{\bm m^\top\bm K^{-1}\bm 1\bm 1^\top\bm K^{-1}\bm 1\bm 1^\top\bm K^{-1}\bm m}{(\bm 1^\top\bm K^{-1}\bm 1)^2}\\
      &\qquad -(1+1)\frac{\bm m^\top\bm K^{-1}\bm 1\bm 1^\top\bm K^{-1}\bm 1}{(\bm 1^\top\bm K^{-1}\bm 1)^2}+ \frac{\bm 1^\top\bm K^{-1}\bm 1}{(\bm 1^\top\bm K^{-1}\bm 1)^2}.
  \end{align*}
  The expression simplifies to
\begin{align}\label{eq:15}
     \MMD(P,P^n)^2
      &= \iint_{\mathcal X\times\mathcal X}k(\xi,\xi^\prime)P(\mathrm d\xi)P(\mathrm d\xi) \\
      &\qquad -\bm m^\top \bm K^{-1}\bm m\nonumber\\
      &\qquad +\frac{\bm m^\top\bm K^{-1}\bm 1\cdot\bm 1^\top\bm K^{-1}\bm m}{\bm 1^\top\bm K^{-1}\bm 1} -2\frac{\bm m^\top\bm K^{-1}\bm 1}{\bm 1^\top\bm K^{-1}\bm 1}+ \frac{1}{\bm 1^\top\bm K^{-1}\bm 1},\nonumber
  \end{align}
  giving the final assertion.
\end{proof}
\begin{remark}[Accumulated costs for involving probability measures]\label{rem:461}
  Comparing the results in Corollaries~\ref{cor:2} and~\ref{cor:412} reveals the costs for insisting on probability measures.
  These costs accumulate to
  \[
    +\frac{{(\bm 1^\top\bm K^{-1}\bm m-1)}^2}{\bm 1^\top\bm K^{-1}\bm 1},\]
  in total.
\end{remark}

\section{Optimal locations}\label{sec:3}
The preceding sections identify the optimal weights (probability weights~$\bm p$, respectively).
This section addresses the optimal facility locations of the quantization problem, that is, the optimal points $\bm x= [x_1,\dots,x_n]$ of the measure $\sum_{i=1}^n \mu_i\,\delta_{x_i}$ or $\sum_{i=1}^n p_i\,\delta_{x_i}$ (respectively).
We reformulate the objective first so that the problem is a pure stochastic optimization problem, that is, an expectation of some specific cost function.
With this, a stochastic gradient method applies.

\subsection{Stochastic optimization}
The following theorem aims at providing a reformulation of the original objective function presented in~\eqref{DMMD} but incorporating the constraint~\eqref{P2}.
The constraints ensure that the resulting measure is a \emph{probability} measure.
By expressing it as a minimization problem involving a cost function that incorporates expectations with respect to random variables $\xi$ and $\xi'$, the problem structure simplifies. This reformulation aids in clearer analysis and practical implementation of the optimization problem.

In what follows, we shall make the dependence on $\bm x= [x_1,\dots,x_n]$ explicit by writing
\begin{equation}\label{eq:16}
	\bm K(\bm x)\coloneqq \bigl(k(x_i,x_j)\bigr)_{i,j=1}^n\ \text{ and }\ k(\xi,\bm x)\coloneqq \bigl(k(\xi,x_1),\dots,k(\xi,x_n)\bigr),
\end{equation}
and\footnote{The subscript in the expectation refers to the random component, $\E_{\xi\sim P}k(\xi,x_i)= \int_{\mathcal X} k(\xi,\xi)\,P(\mathrm d\xi)$.}
\begin{equation}\label{eq:17}
  \bm m(\bm x)\coloneqq \bigl(\E_{\xi\sim P} k(\xi,x_i)\bigr)_{i=1}^n.
\end{equation}
With that, the optimization problem can be reformulated by involving a single expectation.
\begin{theorem}\label{TPS}
  The constrained objective function in~\eqref{P1}--\eqref{P4} can be reformulated as \emph{unconstrained} optimization problem involving the expectation with respect to the product measure $P\otimes P$ by
  \begin{align}\label{eq:18}
    \min_{\bm{x}= [x_1,\dots,x_n]\in \mathcal X^n} \E_{(\xi,\xi')\sim P\otimes P}\bigl[c(\bm{x}, \xi, \xi')\bigr],
  \end{align}
 where
  \begin{align}
    c(\bm x,\xi,\xi^\prime)\coloneqq k(\xi,\xi^\prime)&- k(\xi,\bm x)^\top\bm K(\bm x)^{-1}k(\xi^\prime,\bm x)\nonumber\\
     &+ \frac{1}{\bm 1^\top{\bm K(\bm x)}^{-1}\bm 1}\bigl(1-\bm 1^\top{\bm K(\bm x)}^{-1}k(\xi,\bm x)\bigr)\cdot \bigl(1-\bm 1^\top\bm K(\bm x)^{-1}k(\xi^\prime,\bm x)\bigr) \label{eq:22}
  \end{align}
  is the cost function.
  The expectation in~\eqref{eq:18} is with respect to the probability measure $P\otimes P$, the independent product defined by $(P\otimes Q)(A\times B)\coloneqq P(A)\cdot Q(B)$, and $\bm{x} = [x_1, \ldots, x_n]$ represents the supporting points.
\end{theorem}
\begin{proof}
  The assertion is based on the~\eqref{eq:15}. By the properties of the product measure we have that
  \[ \E_{(\xi,\xi^\prime)\sim P\otimes P}\bigl[k(\xi,\bm x)\,{\bm K(\bm x)}^{-1}\,k(\xi^\prime,\bm x)\bigr]= {\bm m(\bm x)}^\top{\bm K(\bm x)}^{-1}\bm m(\bm x)\]
  for $\bm m(\bm x)\coloneqq \E_{\xi\sim P}k(\xi,\bm x)$.
  Further, note that
  \begin{align*}
 2{\bm m(\bm x)}^\top{\bm K(\bm x)}^{-1}\bm 1&= 2\E_{\xi\sim P}\bigl[k(\xi,\bm x){\bm K(\bm x)}^{-1}\bm 1\bigr] \\
       &= \E_{(\xi,\xi^\prime)\sim P\otimes P}\bigl[\bm 1^\top{\bm K(\bm x)}^{-1}k(\xi,\bm x)+ \bm 1^\top{\bm K(\bm x)}^{-1}k(\xi^\prime,\bm x)\bigr],
\end{align*}
  thus the cost function~$c$, as specified.

  We may state the cost function also explicitly by
\begin{align*}
    c(\bm{x}, \xi, \xi') = k(\xi, \xi') &+ \frac{1}{\bm{1}^\top \bm{K}^{-1} \bm{1}}\\
   & - \sum_{j=1}^n \Bigl(\frac{\bm{1}^\top\bm{K}^{-1}}{\bm{1}^\top\bm{K}^{-1}\bm{1}}\Bigr)_{j}\bigl(k(\xi, x_j)+ k(\xi^\prime,x_j)\bigr) \\
    & -  \sum_{i,j=1}^n \,k(\xi, x_i)\biggl( \bm{K}^{-1}_{ij}-\Bigl(\frac{\bm{K}^{-1}\bm{1}\bm{1}^\top\bm{K}^{-1}}{\bm{1}^\top\bm{K}^{-1}\bm{1}}\Bigr)_{ij}\biggr)  \, k(\xi', x_j).
   \end{align*}
   This expression will be used in numerical implementations below.
\end{proof}
\begin{remark}
  The cost function~\eqref{eq:22} is symmetric, it holds that
  \begin{equation}\label{eq:23}
    c(\bm x,\xi,\xi^\prime)= c(\bm x,\xi^\prime,\xi).
  \end{equation}
  This symmetric reformulation turns out to be useful in the numerical procedure detailed in what follows. 
\end{remark}

\subsection{Stochastic gradient}
Stochastic Gradient Descent (SGD) is a powerful optimization algorithm that updates model parameters using noisy estimates of the gradient derived from randomly selected subsets of data, making it especially suitable for large-scale machine learning tasks (cf.\ among many papers, \citet{RobbinsMonro} or \citet{Norkin2020}).
By iteratively refining parameter estimates based on individual data points rather than the entire dataset, SGD not only accelerates convergence but also helps avoid local minima through its inherent randomness.
On the other hand, Stochastic Approximation provides a broader framework for optimizing a sequence of estimates in the presence of noise, enabling effective learning and adaptation in real-time applications.
Together, these methods leverage the stochastic nature of data to facilitate efficient and robust solutions in diverse fields, including statistical estimation, adaptive control, and reinforcement learning.

To find the optimal set of parameters $\bm{x}$ and $\bm\mu$, e.g., we may employ the stochastic gradient method and update the parameters successively. To this end, the optimization problem~\eqref{eq:5} may be reformulated as
\begin{equation}\label{eq:24}
 \min_{\bm x\in \mathcal X^n}\min_{\bm\mu\in {[0,1]}^n}\iint_{\mathcal X\times\mathcal X}k(\xi,\xi^\prime)-\bm\mu^\top\bigl(k(\xi,\bm x)+k(\xi^\prime,\bm x)\bigr)+ \frac{\bm\mu^\top{\bm K(\bm x)}^{-1}\bm\mu}{{\mu(\mathcal X)}^2}\,\mu(\mathrm d\xi)\mu(\mathrm d\xi^\prime),
\end{equation}
which is an unconstrained optimization problem, the objective involving an expectation with respect to the product measure $\mu\otimes\mu$.

The inner optimization with respect to~$\bm\mu$ has been solved explicitly in Corollary~\ref{cor:2}. By involving~\eqref{eq:7} explicitly, the inner optimization problem in~\eqref{eq:24} has an explicit solution, and the outer optimization simplifies to
\[ \min_{\bm x\in \mathcal X^n}\iint_{\mathcal X\times\mathcal X}k(\xi,\xi^\prime)- k(\xi,\bm x)^\top\bm K(\bm x)^{-1}k(\xi^\prime,\bm x)\,\mu(\mathrm d\xi)\mu(\mathrm d\xi^\prime),\]
which is on the location parameters $\bm x\in\mathcal X^n$ only, and thus a significant simplification over~\eqref{eq:24}.

For the quantization of probability measures, the cost function is given in~\eqref{eq:22}, the optimization problem is~\eqref{eq:18}.

Algorithm~\ref{alg:SGD} details the stochastic gradient method for the problem~\eqref{eq:18} in Theorem~\ref{TPS} to identify the discrete measure approximating a probability measure in the smallest maximum mean discrepancy distance.

\begin{remark}[Computational expenses to evaluate the cost function]
  Algorithm~\ref{alg:SGD} requires evaluating the cost function~\eqref{eq:25}. 
  Note that computing the cost function via~\eqref{eq:22} requires solving the systems
  \begin{equation}\label{eq:28}
        \bm K(\bm x)\bm w= k(\xi,\bm x)\text{ and }\bm K(\bm x)\bm w^\prime= k(\xi^\prime,\bm x)
  \end{equation}
  for $\bm w$ and $\bm w^\prime$ in~\eqref{eq:22}.
  This is the most expensive step in Algorithm~\ref{alg:SGD}.

  However, these costs can be reduced.
  Indeed, consider the cost function
  \begin{align}
    c^\prime(\bm x,\xi,\xi^\prime)\coloneqq k(\xi,\xi^\prime)&- k(\xi^\prime,\bm x)^\top\bm K(\bm x)^{-1}k(\xi,\bm x)\nonumber\\
     &+ \frac{\bigl(1-\bm 1^\top\bm K(\bm x)^{-1}k(\xi,\bm x)\bigr)^2}{\bm 1^\top\bm K(\bm x)^{-1}\bm 1}.\label{eq:27}
  \end{align}
  As the expectation is with respect to the independent product measure $P\otimes P$ in~\eqref{eq:18}, it holds that
  \[ \E_{(\xi,\xi^\prime)\sim P\otimes P}\bigl[c(\bm x,\xi,\xi^\prime)\bigr]= \E_{(\xi,\xi^\prime)\sim P\otimes P}\bigl[c^\prime(\bm x,\xi,\xi^\prime)\bigr].\]
  In evaluating $c^\prime$, only the first system in~\eqref{eq:28} needs to be solved, thus reducing the total expenses by $\nicefrac12$ and accelerating the algorithm.
  In contrast to~$c$, the cost function $c^\prime$ is not symmetric any longer (cf.~\eqref{eq:23}).
\end{remark}

\begin{algorithm}[t]
\caption{Stochastic gradient descent for quantization of a probability measure~$P$\label{alg:SGD}}
\begin{algorithmic}[1]
\STATE \textbf{input:} probability measure~$P$, initial support points $\bm x_0$, kernel function $k(\cdot, \cdot)$
\STATE \textbf{output:} quantization measure $P^n= \sum_{i=1}^n p_i\,\delta_{x_i}$ with $n$ optimal locations $\bm x^*=[x_1,\dots,x_n]$ and weights~$\bm p= [p_1,\dots,p_n]$, and the maximum mean discrepancy $\MMD(P,\, P^n)$

\STATE initialize $\bm x \leftarrow \bm x_0$, $\bm m\leftarrow \bm 0$, $t\leftarrow 0$ \hfill\COMMENT{initialize the support points, the mean vector and the iteration count~$t$}

\REPEAT
    \STATE generate two random, independent samples $\xi$ and $\xi^\prime$ from $P$
    
    \STATE update the average maximum mean discrepancy with cost function~\eqref{eq:22}, \hfill\COMMENT{cf.\ Corollary~\ref{cor:412}}
    \begin{equation}\label{eq:25}
      \MMD^2\leftarrow \frac1{t+1}\bigl(t\cdot \MMD^2+ c(\bm x, \xi, \xi^\prime, \bm m)\bigr)
    \end{equation}
    \STATE compute the gradient of the cost function employed in~\eqref{eq:25},
    \begin{equation}\label{eq:26}
        \nabla c \leftarrow \nabla_{\bm x}\, c(\bm x, \xi, \xi^\prime, \bm m)
    \end{equation}
    \STATE $\bm x \leftarrow \bm x - \eta_t \cdot \nabla c$ with learning rate $\eta_t \leftarrow \frac{1.0}{100 + t}$ \hfill\COMMENT{update the support points. \\
    \hfill The learning rate $\eta_t$ satisfies $\sum_{t>0}\eta_t=\infty$, and $\sum_{t>0}\eta_t^2<\infty$.}
    \STATE $ \bm m \leftarrow \frac1{t+1}\bigl(t\cdot \bm m + \frac12\bigl(k(\xi,\bm x)+k(\xi^\prime,\bm x)\bigr)\bigr)$ \hfill\COMMENT{cf.~\eqref{eq:17}}
    \STATE $t\leftarrow t+1$ \hfill\COMMENT{update the iteration count}
\UNTIL desired approximation precision reached

\STATE compute the optimal weights $\bm p$ with~\eqref{eq:8}

\STATE \textbf{return} $(\sum_{i=1}^n p_i\,\delta_{x_i},\ \MMD)$\qquad\qquad\hfill\COMMENT{return the quantization measure with optimal locations and the corresponding maximum mean discrepancy}
\end{algorithmic}
\end{algorithm}
\subsection{Non-negativity constraints}\label{sec:6}
The results in the preceding section on the quantization of probability measures address all constraints in~\eqref{eq:4}, except the non-negativity constraint~\eqref{P3}. Figure~\ref{fig:b22} below (page~\pageref{fig:b22}) demonstrates that these non-negativity constraints are occasionally necessary.
For fixed quantizers $\bm x= [x_1,\dots,x_n]$, these constraints are easy to handle, simply by removing~$x_j$ from the list $\bm x$ for every index~$j$ with $p_j<0$, that is to activate the non-negativity constraint~\eqref{P3}.

Both, the simplex constraints~\eqref{P2} and the non-negativity constraints~\eqref{P3} can be incorporated in a single stochastic optimization reformulation as in~\eqref{eq:18}. To this end consider the simplex
\[ \mathcal S\coloneqq \{(p_1,\dots,p_n)\colon p_1\ge0,\dots, p_n\ge 0\text{ and }p_1+\dots + p_n=1\},\]
and the projection
\begin{equation}\label{eq:33}
	\pi(\bm\mu)\coloneqq \argmin_{\bm p\in\mathcal S}\|\bm\mu-\bm p\|\quad (\in\mathcal S).
\end{equation}
The crucial property of $\pi$ employed below is that $\pi(\bm\mu)= \bm\mu$ for $\bm\mu\in\mathcal S$.
With that, the complete constrained optimization problem~\eqref{P1}--\eqref{P4} can be recast as global, unconstrained optimization problem of the form
\begin{equation}\label{eq:31}
	\min_{(\bm x,\bm\mu) \in \mathcal X^n\times  \mathbb R^n} \E_{(\xi,\xi^\prime)\sim P\otimes P}\bigl[c^{\prime\prime}(\bm x,\pi(\bm\mu),\xi,\xi^\prime)\bigr]+ \|\bm\mu-\pi(\bm\mu)\|,
\end{equation}
where
\begin{align}
  c^{\prime\prime}(\bm x,\bm\mu,\xi,\xi^\prime)& \coloneqq k(\xi,\xi^\prime)- \bm\mu^\top k(\bm x,\xi)- \bm\mu^\top k(\bm x,\xi^\prime)+\bm\mu^\top\bm K(\bm x)\,\bm\mu\label{eq:32}.
\end{align}
The extra term in~\eqref{eq:31} is a penalty if $\mu\notin \mathcal S$, while $\pi(\bm\mu)$ in~\eqref{eq:32} ensures the constraints~\eqref{P2} and~\eqref{P3} simultaneously. This modification thus guarantees the optimal probability to be in the simplex, $\pi(\bm\mu)\in \mathcal S$.

The problem~\eqref{eq:31}, however, is an optimization problem on $\mathcal X^n\times \mathbb R^n$, and does not take advantage of the optimal weights directly derived in~\eqref{eq:8}.
We refer to~\citet{Norkin2025} (cf.\ also \citet{NorkinPichler}) for the optimization problem~\eqref{eq:31} with constraints governed by a function as in~\eqref{eq:33}.
\bigskip

To take advantage of the optimal probabilities~\eqref{eq:8} and to automate active and inactive non-negativity constraints for optimization on $\bm p$ and $\bm x$, let
\[
  J\coloneqq \{j\colon p_j\geq 0 \} \text { and } J^\mathsf{c}= \{ j\colon p_j < 0 \}\]
be the sets of indices corresponding to non-negative and negative probabilities, respectively.

We define a modified Lagrangian function to optimize probabilities while ensuring non-negative weights. The following theorem presents the necessary conditions for optimality derived from this formulation.

\begin{theorem}
  The solution of the optimization problem~\eqref{eq:4} with Lagrangian function
\begin{align*}
L(p_1, \ldots, p_n, \lambda, \mu_1, \ldots, \mu_m) \coloneqq \| P_k- P_k^n \|_k^2 + \lambda \left(\sum_{i=1}^n p_i- 1\right) + \sum_{i \in J^\mathsf{c}} \mu_i\,p_i
\end{align*}
  with $m\coloneqq \#J^\mathsf{c}$, the number of active non-negativity constraints,
  is given by
\begin{align*}
\begin{pmatrix}
\bm{K}_{J, J}  &\bm{1}_{J} & 0\\
\bm{K}_{J^\mathsf{c}, J}  &\bm{1}_{J^\mathsf{c}} & I_{J^\mathsf{c}}\\
\bm{1}_{J}^\top &0&0
\end{pmatrix}
\begin{pmatrix}
\bm{p}_{J} \\
\lambda/2\\
\mu/2
\end{pmatrix}=
\begin{pmatrix}
\bm{m}_{J}\\
\bm{m}_{J^\mathsf{c}}\\
1 
\end{pmatrix}
\end{align*}
  and $\bm p_{J^\mathsf{c}}= \bm 0$.

  Here, $\bm{1}_{J}$ and $\bm{1}_{J^\mathsf{c}}$ are column vectors of ones with dimensions equal to the sizes of the sets $J$ and $J^\mathsf{c}$, respectively, the symbols 0 and 1 represent scalars, and $I_{J^\mathsf{c}}$ is the identity matrix of size equal to the dimension of $J^\mathsf{c}$.
\end{theorem}
In the preceding theorem, $ \lambda $ is the Lagrange multiplier enforcing the equality constraint $\sum_{i=1}^n p_i = 1 $, while $ \mu_i $ is the Lagrange multiplier for the inequality constraints, where $\mu_i= 0$ for each $i \in J $ and $\mu_i > 0 $ for each $i \in J^\mathsf{c}$.
The term $ \sum_{i \in J^\mathsf{c}} \mu_i\,p_i$ introduces a penalty for negative probabilities, thereby discouraging $ p_i $ from being negative and effectively enforcing the condition $p_i\geq 0$ during optimization. Thus, the formulation ensures that the optimization process respects the constraints on the probabilities while minimizing the Lagrangian.
\begin{proof}
  We first derive the necessary conditions for optimality by computing the derivatives of the Lagrangian with respect to
  $\lambda$, $p_\ell$ and $\mu_\ell$ (for $\ell \in J^\mathsf{c}$):
  \begin{itemize}
    \item The derivative with respect to $\lambda$ is
 \begin{align*}
    \frac{\partial L}{\partial \lambda} = \sum_{i=1}^n p_i - 1 = 0.
  \end{align*}
    \item The derivative with respect to $\mu_\ell$ (complementary slackness) if $\ell \in J^\mathsf{c}$ is given by
\begin{align*}
    \frac{\partial L}{\partial \mu_\ell} = p_\ell.
\end{align*}
	\item Derivative with respect to $p_\ell$: if $\ell \in J^\mathsf{c}$,
 \begin{align}\label{Jc}
    \frac{\partial L}{\partial p_\ell} = -2\int_{\mathcal X} k(x,x_\ell) \, \mu(\mathrm{d}x) + 2 \sum_{i=1}^n k(x_i,x_\ell) p_i + \lambda + \mu_\ell = 0,
\end{align}
		if ${\ell \in J}$,
\begin{align}\label{J}
    \frac{\partial L}{\partial p_\ell} = -2\int_{\mathcal X} k(x,x_\ell) \, \mu(\mathrm{d}x) + 2 \sum_{i=1}^n k(x_i,x_\ell) p_i + \lambda = 0.
\end{align}
   \end{itemize}

  By multiplying each term of equation~\eqref{Jc} by $p_\ell$ that ${\ell \in J^\mathsf{c}}$ and summing this equation over ${\ell \in J^\mathsf{c}}$, we get
  \begin{align}\label{JC1}
 \lambda \sum_{\ell \in J^\mathsf{c}}p_\ell = 2 \sum_{\ell \in J^\mathsf{c}} p_\ell  \int_{\mathcal X} k(x,x_\ell) \, \mu(\mathrm{d}x)  - 2\sum_{\ell \in J^\mathsf{c}} \sum_{i=1}^n k(x_i,x_\ell) p_i p_\ell-\sum_{\ell \in J^\mathsf{c}}p_\ell \mu_\ell.
 \end{align}  
  Moreover, multiplying each term of equation~\eqref{J} by $p_\ell$, and summing this equation over ${\ell \in J}$, we get
\begin{align}\label{J1}
 \lambda \sum_{\ell \in J}p_\ell = 2 \sum_{\ell \in J} p_\ell  \int_\mathcal{X} k(x,x_\ell) \, \mu(\mathrm{d}x)  - 2\sum_{\ell \in J} \sum_{i=1}^n k(x_i,x_\ell) p_i p_\ell.
 \end{align}  
  By adding equations~\eqref{JC1} and~\eqref{J1}, we obtain
  \begin{align*}
 \nonumber \lambda \sum_{\ell \in J^\mathsf{c}}p_\ell&+\lambda \sum_{\ell \in J}p_\ell=
 \nonumber\\=& 2 \sum_{\ell \in J^\mathsf{c}} p_\ell  \int_\mathcal{X} k(x,x_\ell) \, \mu(\mathrm{d}x)  - 2\sum_{\ell \in J^\mathsf{c}} \sum_{i=1}^n k(x_i,x_\ell) p_i p_\ell-\sum_{\ell \in J^\mathsf{c}}p_\ell \mu_\ell
 \nonumber\\+& 2 \sum_{\ell \in J} p_\ell  \int_{\mathcal X} k(x,x_\ell) \, \mu(\mathrm{d}x)  - 2\sum_{\ell \in J} \sum_{i=1}^n k(x_i,x_\ell) p_i p_\ell.
 \end{align*}    
  Thus, we obtain
\begin{align*}
  \lambda \Bigl( & \sum_{\ell \in J^\mathsf{c}}p_\ell+ \sum_{\ell \in J}p_\ell \Bigr) 
    = 2 \sum_{\ell \in J^\mathsf{c}} p_\ell  \int_{\mathcal X} k(x,x_\ell) \, \mu(\mathrm{d}x) +2 \sum_{\ell \in J} p_\ell  \int_{\mathcal X} k(x,x_\ell) \,\mu(\mathrm{d}x) \\
   &\ -2 \left( 
 \sum_{\ell \in J^\mathsf{c}} \sum_{i=1}^n k(x_i,x_\ell) p_i p_\ell +
 \sum_{\ell \in J} \sum_{i=1}^n k(x_i,x_\ell) p_i p_\ell\right) 
 - \sum_{\ell \in J^\mathsf{c}}p_\ell \mu_\ell
 \\&= 2 \sum_{\ell \in J^\mathsf{c}} p_\ell  \int_{\mathcal X} k(x,x_\ell) \, \mu(\mathrm{d}x) +2 \sum_{\ell \in J} p_\ell  \int_{\mathcal X} k(x,x_\ell) \, \mu(\mathrm{d}x)
  \\&\ -2 \left( 
 \sum_{\ell \in J^\mathsf{c}} \sum_{i \in J^\mathsf{c}} k(x_i,x_\ell) p_i p_\ell +\sum_{\ell \in J^\mathsf{c}} \sum_{i \in J} k(x_i,x_\ell) p_i p_\ell +
 \sum_{\ell \in J} \sum_{i \in J} k(x_i,x_\ell) p_i p_\ell+\sum_{\ell \in J} \sum_{i \in J^\mathsf{c}} k(x_i,x_\ell) p_i p_\ell\right)
  \\&\ - \sum_{\ell \in J^\mathsf{c}}p_\ell \mu_\ell.
 \end{align*}            
  Therefore, employing the compact notation,
  \begin{align}\label{Total}
 \nonumber \lambda \Bigl(\bm{p}_{J^\mathsf{c}}^\top\bm{1}_{J^\mathsf{c}}&+ \bm{p}_{J}^\top\bm{1}_J \Bigr) =\lambda \left( \sum_{i=1}^n p_i \right) 
 = 2 \left( \bm{p}_{J^\mathsf{c}}^\top \bm{m}_{J^\mathsf{c}}+\bm{p}_{J}^\top \bm{m}_{J}\right) 
 \nonumber\\-& 2 \left( 
 \bm{p}_{J^\mathsf{c}}^\top\,\bm{K}_{J^\mathsf{c}, J^\mathsf{c}}\,\bm{p}_{J^\mathsf{c}}+\bm{p}_{J^\mathsf{c}}^\top\,\bm{K}_{J^\mathsf{c}, J}\,\bm{p}_{J} +
 \bm{p}_{J}^\top\,\bm{K}_{J, J}\,\bm{p}_J+\bm{p}_J^\top\,\bm{K}_{J, J^\mathsf{c}}\,\bm{p}_{J^\mathsf{c}}\right) 
 - \bm{p}_{J^\mathsf{c}}^\top \mu.
 \end{align}       
  Since $\sum_{i=1}^n p_i = 1$, we can derive from equation~\eqref{Total} that
  \begin{align*}
  \lambda &= 2 \left( \bm{p}_{J^\mathsf{c}}^\top \bm{m}_{J^\mathsf{c}}+\bm{p}_{J}^\top \bm{m}_{J}\right) \\
 &\ -
 2 \left( 
 \bm{p}_{J^\mathsf{c}}^\top\,\bm{K}_{J^\mathsf{c}, J^\mathsf{c}}\,\bm{p}_{J^\mathsf{c}}+\bm{p}_{J^\mathsf{c}}^\top\,\bm{K}_{J^\mathsf{c}, J}\,\bm{p}_{J} +
 \bm{p}_{J}^\top\,\bm{K}_{J, J}\,\bm{p}_{J}+\bm{p}_{J}^\top\,\bm{K}_{J, J^\mathsf{c}}\,\bm{p}_{J^\mathsf{c}}\right) 
 - \bm{p}_{J^\mathsf{c}}^\top \mu.
 \end{align*}       
  From equations~\eqref{JC1} and~\eqref{J1}, we have 
  \begin{align*}
 \lambda \bm{p}_{J^\mathsf{c}}^\top\bm{1}_{J^\mathsf{c}}= 2 \bm{p}_{J^\mathsf{c}}^\top \bm{m}_{J^\mathsf{c}}  - 2\bm{p}_{J^\mathsf{c}}^\top \bm{K}_{J^\mathsf{c}, J^\mathsf{c}}\bm{p}_{J^\mathsf{c}} -2\bm{p}_{J^\mathsf{c}}^\top \bm{K}_{J^\mathsf{c}, J}\bm{p}_{J} -\bm{p}_{J^\mathsf{c}}^\top \mu,
 \end{align*}  
  and  
  \begin{align*}
 \lambda \bm{p}_{J}^\top\bm{1}_{J}= 2 \bm{p}_{J}^\top \bm{m}_{J}  - 2\bm{p}_{J}^\top \bm{K}_{J, J^\mathsf{c}}\bm{p}_{J^\mathsf{c}} -2\bm{p}_{J}^\top \bm{K}_{J, J}\bm{p}_{J},
 \end{align*}  
  respectively.
  Thus, \begin{align}\label{First}
 \lambda \bm{1}_{J^\mathsf{c}}= 2 \bm{m}_{J^\mathsf{c}}  - 2\bm{K}_{J^\mathsf{c}, J^\mathsf{c}}\bm{p}_{J^\mathsf{c}} -2 \bm{K}_{J^\mathsf{c}, J}\bm{p}_{J} -  I_{J^\mathsf{c}}\mu,
 \end{align}  
  and \begin{align}\label{Second}
 \lambda \bm{1}_{J}= 2  \bm{m}_{J}  - 2 \bm{K}_{J, J^\mathsf{c}}\bm{p}_{J^\mathsf{c}} -2 \bm{K}_{J, J}\bm{p}_{J},
 \end{align}  
  Considering equations~\eqref{First},~\eqref{Second}, and under the condition that $\bm{p}_{J^\mathsf{c}}=0$, we obtain: 
 \begin{align*}
 \begin{pmatrix}
 \bm{K}_{J, J} &\bm{K}_{J, J^\mathsf{c}} &\bm{1}_{J} & 0\\
 \bm{K}_{J^\mathsf{c}, J} &\bm{K}_{J^\mathsf{c}, J^\mathsf{c}} &\bm{1}_{J^\mathsf{c}} & I_{J^\mathsf{c}}\\
 \bm{1}_{J}^\top & \bm{1}_{J^\mathsf{c}}^\top&0&0\\
 0&  I_{J^\mathsf{c}} &0&0
 \end{pmatrix}
 \begin{pmatrix}
 \bm{p}_{J} \\
 \bm{p}_{J^\mathsf{c}} \\
 \lambda/2\\
 \mu/2\\
 \end{pmatrix}=
 \begin{pmatrix}
 \bm{m}_{J}\\
 \bm{m}_{J^\mathsf{c}}\\
 1 \\
 0
 \end{pmatrix}.
 \end{align*}                   
  The assertion of the theorem reformulates the latter display.
\end{proof}

\subsection{Quantization of the normal distribution}
This section provides a deterministic solution for optimizing the MMD squared objective using a normal kernel, eliminating the need for stochastic approximation.
By deriving closed-form expressions and employing a modified Lagrangian, we ensure the normalization constraint. The optimality conditions lead to a set of equations for efficient computation, offering a precise solution with broader applications in kernel methods and optimization theory.

In what follows, consider the normal distribution
\[ P\sim \mathcal N(\mu,\sigma^2) \]
together with the kernel
\[	k(x,y)\coloneqq k_\ell(|y-x|),
          \]
where
\begin{equation}\label{eq:21}
	k_\ell(z)\coloneqq \frac1{\sqrt{2\pi\ell^2}}e^{-\frac{z^2}{2\ell^2}}
\end{equation}
has differing bandwidth~$\ell$.
This objective function involves a specific integral related to the normal kernel. To find deterministic solutions rather than relying on stochastic methods, we derive a closed-form expression for this integral, leading to a concrete optimization problem.

The embedding of the measure $P$ in $\mathcal H_k$ (cf.~\eqref{eq:1}) is found by the convolution theorem for normally distributed random variables, it is given explicitly by the function
\[ P_k(x)= \int_{\mathbb R} k(x,y)P(\mathrm{d}y)= \frac1{\sqrt{2\pi(\sigma^2+\ell^2)}}e^{-\frac{{(x-\mu)}^2}{2(\sigma^2+\ell^2)}}\in \mathcal H_k.\]
The approximating measure thus has the embedding
\[ P^n_k(x)= \sum_{i=1}^n p_i\cdot \frac1{\sqrt{2\pi\ell^2}}e^{\frac{{(x-x_i)}^2}{2\ell^2}}\in \mathcal H_k.\]
In the same way, the vector $\bm{m}$ is given explicitly via~\eqref{eq:9} by the entries
\[ m_i= P_k(x_i)= \E_{\xi\sim P}k(x_i,\xi)=\frac1{\sqrt{2\pi(\sigma^2+\ell^2)}}e^{-\frac{{(x_i-\mu)}^2}{2(\sigma^2+\ell^2)}},\quad i=1,\dots,n,\]
that is, $m_i= P_k(x_i)$.
With that, the approximation quality is available explicitly using~\eqref{eq:7} by
\begin{align*}
  \MMD{(P,\mu^n)}^2&= \|P_k- \mu_k^n\|^2\\
             &= \frac1{\sqrt{2\pi(2\sigma^2+\ell^2)}}- \frac1{2\pi(\sigma^2+\ell^2)}\sum_{i,j=1}^n e^{-\frac{(x_i-\mu)^2+(x_j-\mu)^2}{2(\sigma^2+\ell^2)}}\bm K(\bm x)^{-1}_{ij},
\end{align*}
and the optimal weights are $\mu= {\bm{K}(\bm x)}^{-1}\,\bm{m}(\bm x)$ by Proposition~\ref{prop:180}.

Deterministic optimization methods can handled this objective, and stochastic gradient methods are not needed any longer for this special case.

The method extends from~$\mathbb R$ to normal distributions on $\mathbb R^d$ analoguously.
\begin{remark}
    Note as well that
    \[	k^\prime(z)= -\frac{z}{\ell^2}k(z) \text{ and } P^\prime_k(x)= -\frac{x-\mu}{\ell^2+\sigma^2} P_k(x)\]
    so that computing the gradient in~\eqref{eq:26} (Algorithm~\ref{alg:SGD}) simplifies significantly by applying Stein’s lemma.
\end{remark}

\section{Numerical illustration}\label{sec:Results}
The optimization results validate the effectiveness of our approach in achieving both accurate and efficient solutions.
This section details our findings, demonstrating how the Algorithm~\ref{alg:SGD} performs in terms of convergence behavior, solution accuracy, and computational efficiency. Additionally, the results confirm that our methods are well-suited for practical applications, offering robust performance across a range of optimization problems.

The following computations are executed by employing the Matérn covariance function
\[ k(x,y)\coloneqq C_{\ell,\nu}(|y-x|),\]
where
\begin{equation}\label{eq:29}
  C_{\ell,\nu}(d)\coloneqq \frac{2^{-\nu}\sqrt{2\nu}}{\ell\sqrt\pi \Gamma\bigl(\nu+\frac12\bigr)}\cdot\left(\sqrt{2\nu}\frac{d}{\ell}\right)^\nu  K_\nu\biggl(\sqrt{2\nu}\frac{d}\ell\biggr).
\end{equation}
The function~$K_\nu$ is the modified Bessel function of the second kind.
The kernel function~$k$ provides a positive definite kernel with bandwidth parameter~$\ell$, while~$\nu$ is the smoothness parameter.
For $\nu=\infty$, the kernel is the normal kernel (cf.~\eqref{eq:21}), while for $\nu= \nicefrac12$, the kernel is the density of the (symmetric) exponential distribution.
The constant in the kernel~\eqref{eq:29} is chosen so that $\int_{-\infty}^\infty k(x,y)\,\mathrm{d}y=1$ for every $x\in \mathbb R$.

Figure~\ref{fig:ab} displays the quantification of the normal distribution, Table~\ref{tab:1} provides the corresponding approximation quality, measuring the maximum mean discrepancy.
Figure~\ref{fig:ab2} and Table~\ref{tab:uniform_distribution_results} display the results for the uniform distribution, and Figure~\ref{fig:ab3} and Table~\ref{tab:exponential_distribution_results} provide the results for the exponential distribution.

\begin{figure}[ht]
  \centering
  \subfloat[][bandwidth $\ell=0.1$, smoothness $\nu= \nicefrac12$]
  {\includegraphics[width=0.49\textwidth]{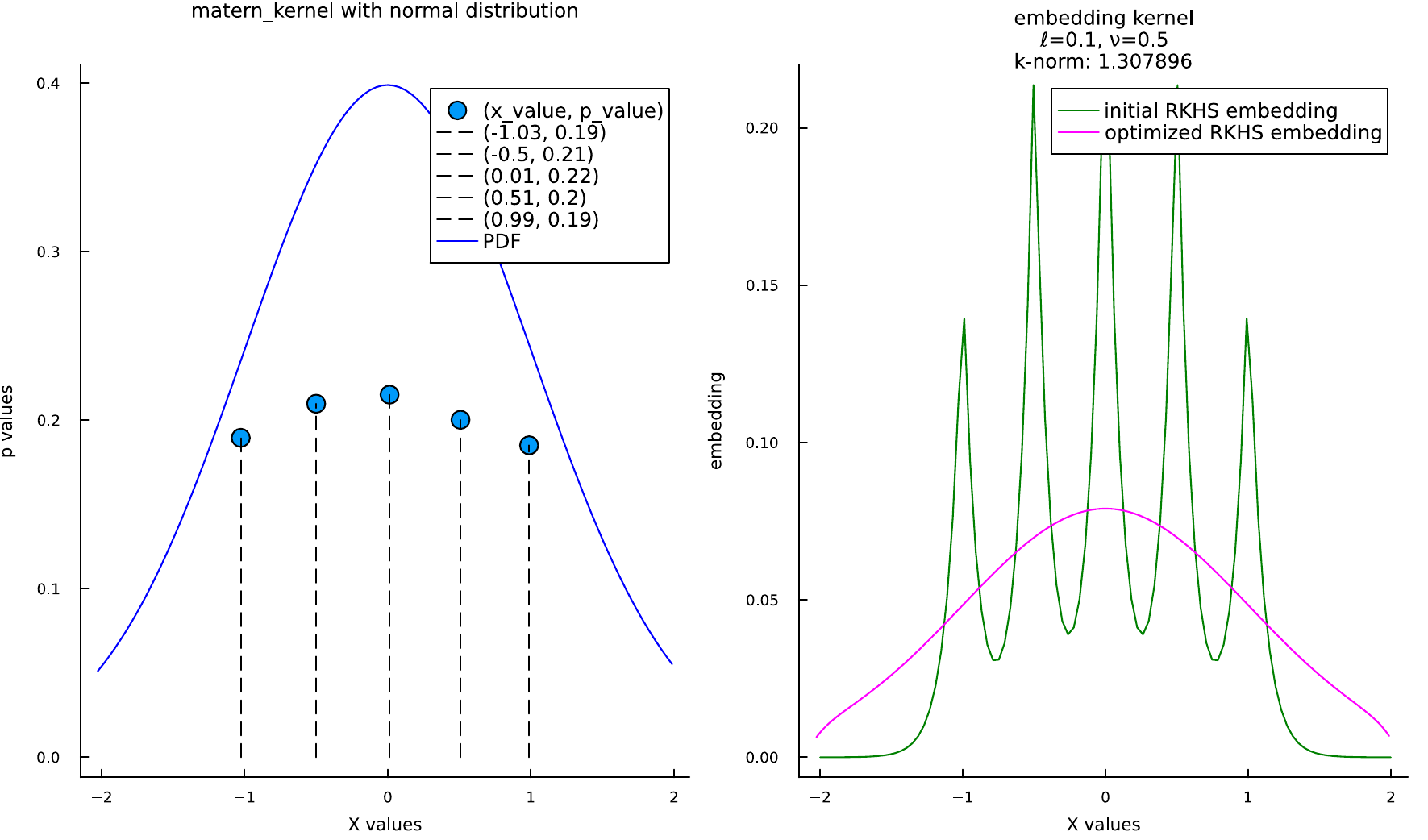}\label{fig:a}}
  \hfill
  \subfloat[][$\ell= 0.1$, $\nu= \infty$]
  {\includegraphics[width=0.49\textwidth]{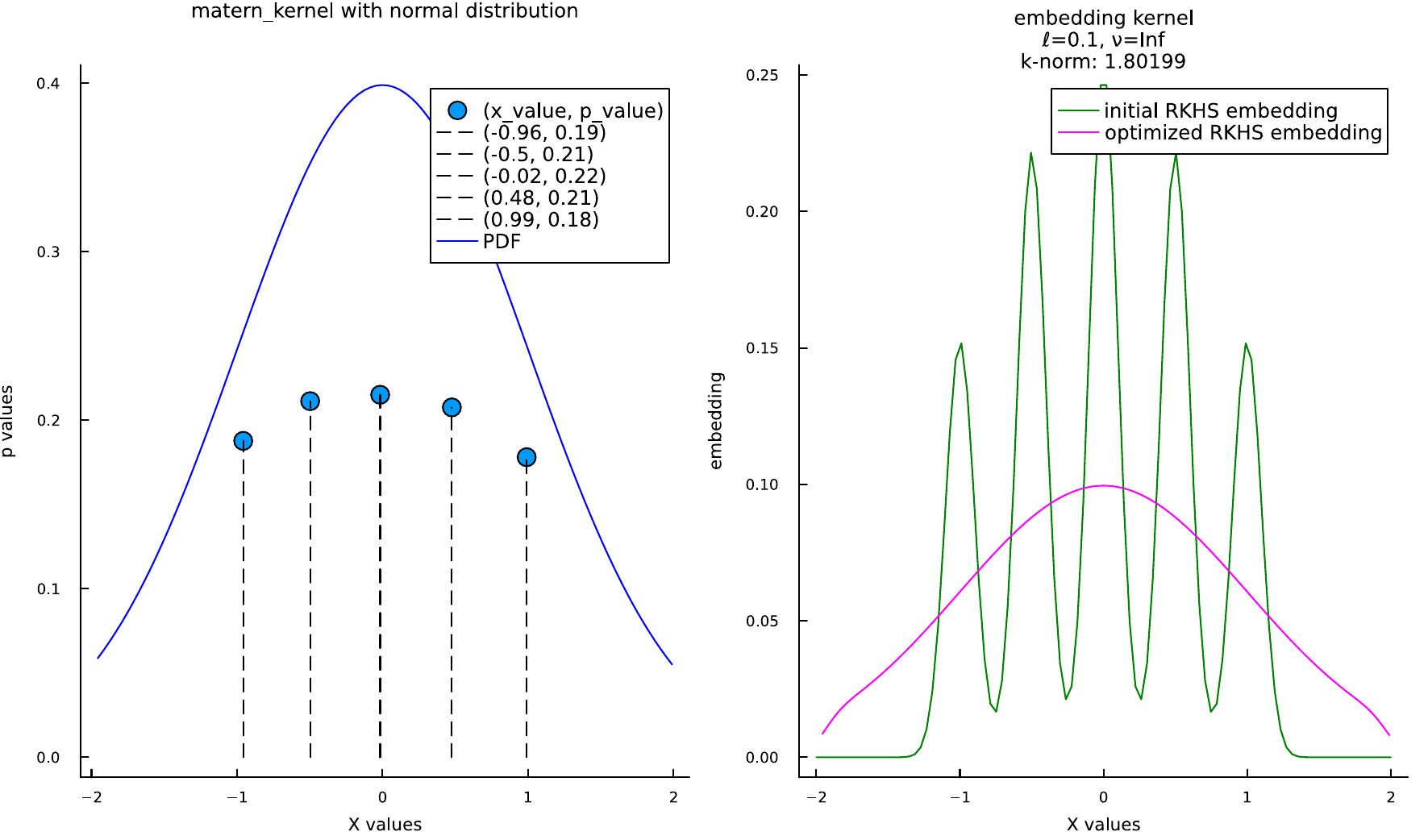}\label{fig:b}}\\
  \subfloat[][bandwidth $\ell=0.5$, smoothness $\nu= \nicefrac12$]
  {\includegraphics[width=0.49\textwidth]{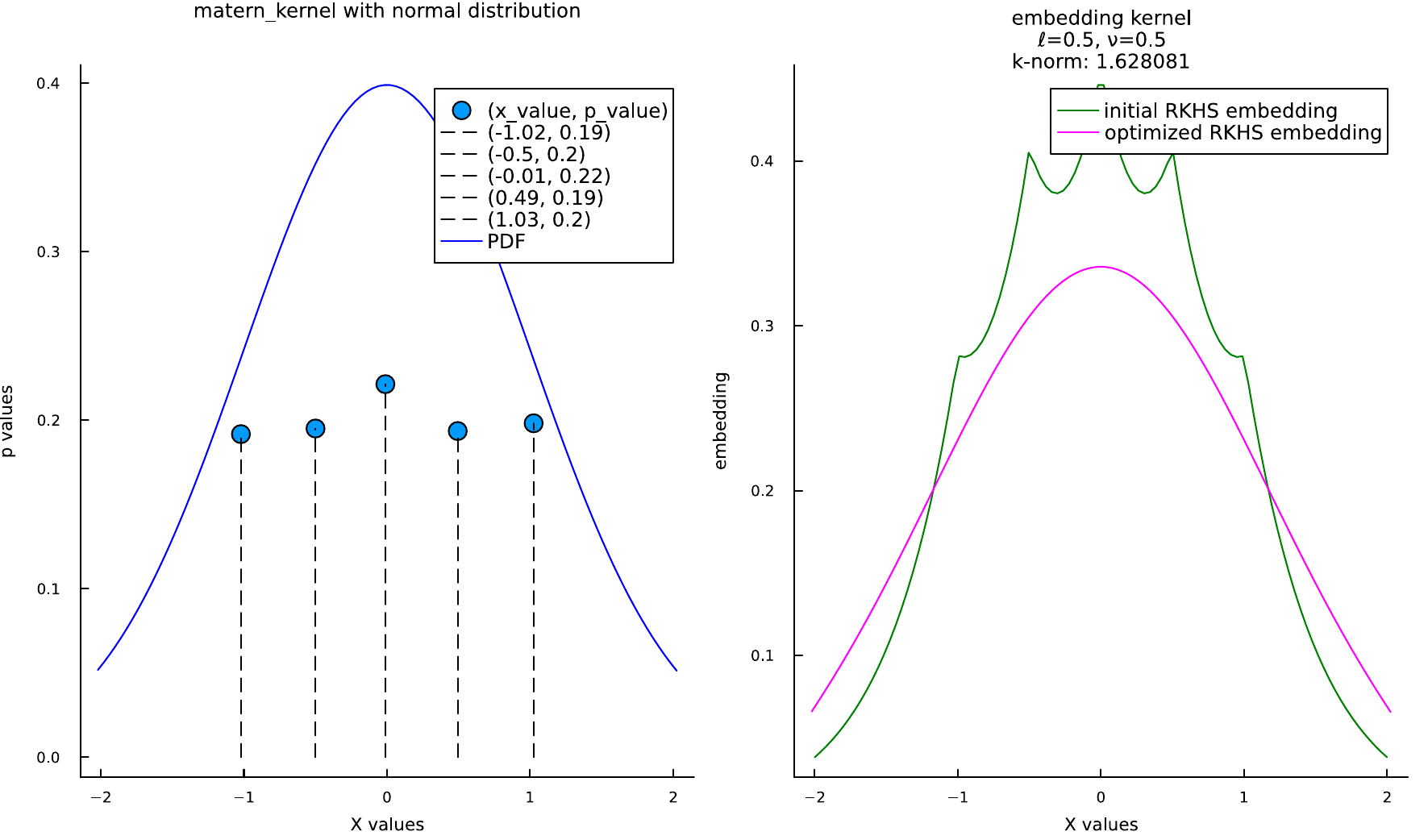}\label{fig:a1}}
  \hfill
  \subfloat[][$\ell= 0.5$, $\nu= \infty$]
  {\includegraphics[width=0.49\textwidth]{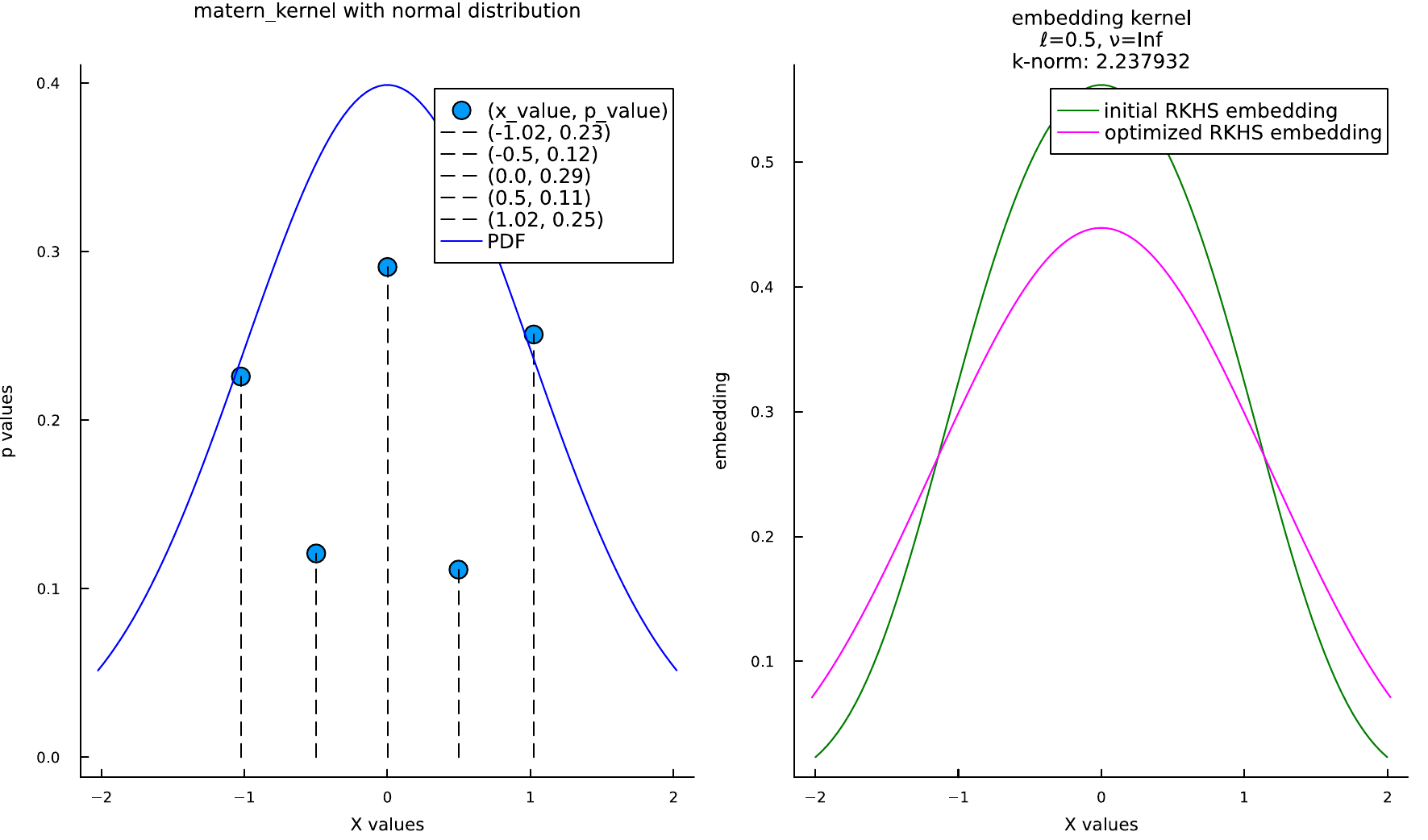}}
  \caption{\label{fig:ab}Approximation of the normal distribution ($P$) with Matérn kernels with different bandwidth~$\ell$ and smoothness~$\nu$ with $n=5$ support points. Density of~$P$ with probability mass function (left), and their embeddings~$P_k,\ P_k^5\in \mathcal H_k$ (right); cf.\ also Table~\ref{tab:normal_distribution_results}}
\end{figure}





\begin{table}[H] 
    \centering
    \begin{tabular}{ccc}
        \toprule
        bandwidth~$\ell$ & smoothness $\nu$ & $\MMD$ distance\\
        \midrule
        0.1 & 0.5 & 1.30790\\
        0.1 & 2.5 & 1.11556\\
        0.1 & $\infty$ & 1.80199\\
        0.5 & 0.5 & 1.62808\\
        0.5 & 2.5 & 1.38039\\
        0.5 & $\infty$ & 2.23793\\
        \bottomrule
    \end{tabular}
    \caption{\label{tab:1}Approximations obtained for the normal distribution, cf.\ Figure~\ref{fig:ab}}\label{tab:normal_distribution_results}
\end{table}

\begin{figure}[ht]
  \centering
  \subfloat[][$\ell=0.1$, $\nu=\nicefrac12$]
  {\includegraphics[width=0.49\textwidth]{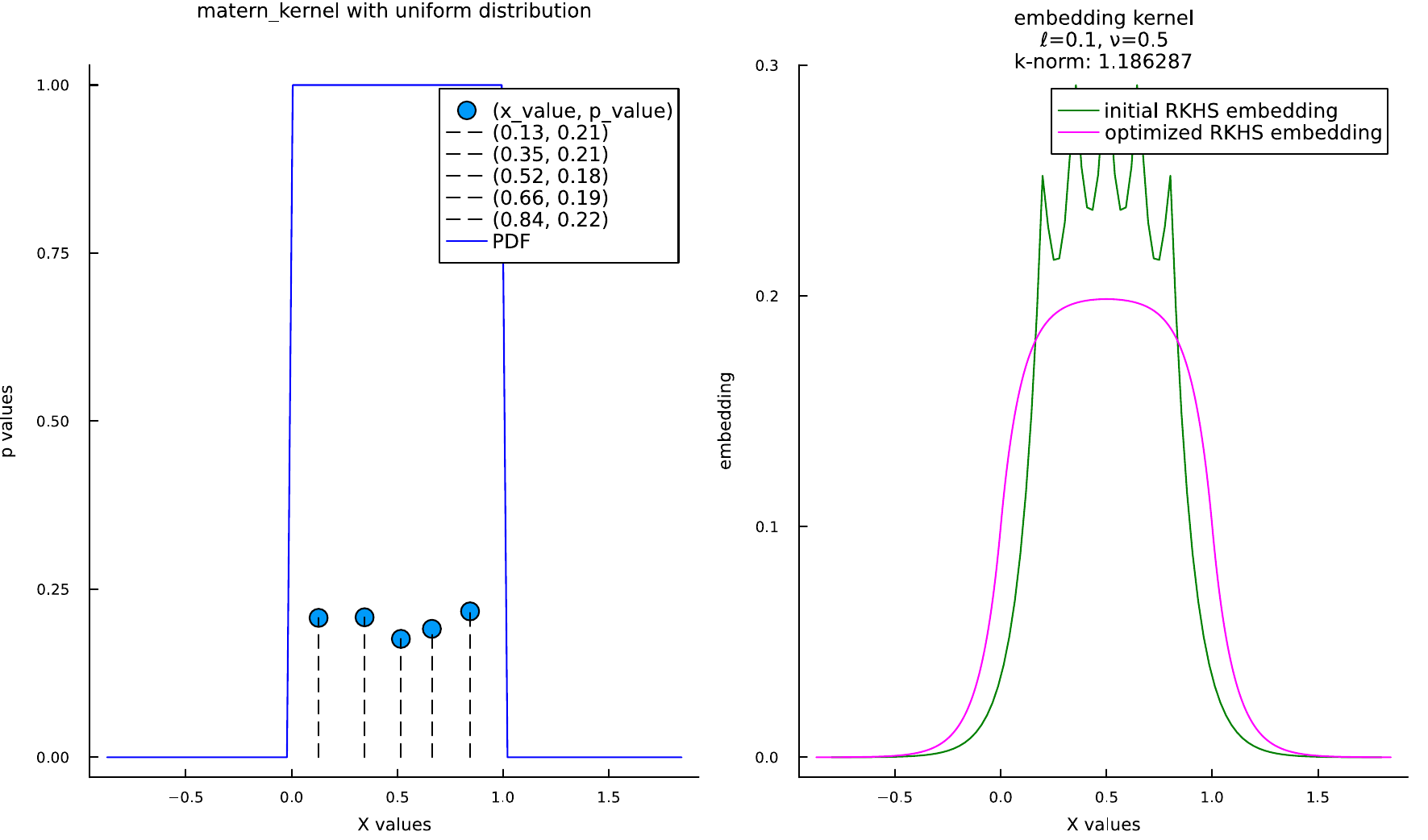}\label{fig:a21}}
  \hfill
  \subfloat[][$\ell=0.1$, $\nu=\infty$]
  {\includegraphics[width=0.49\textwidth]{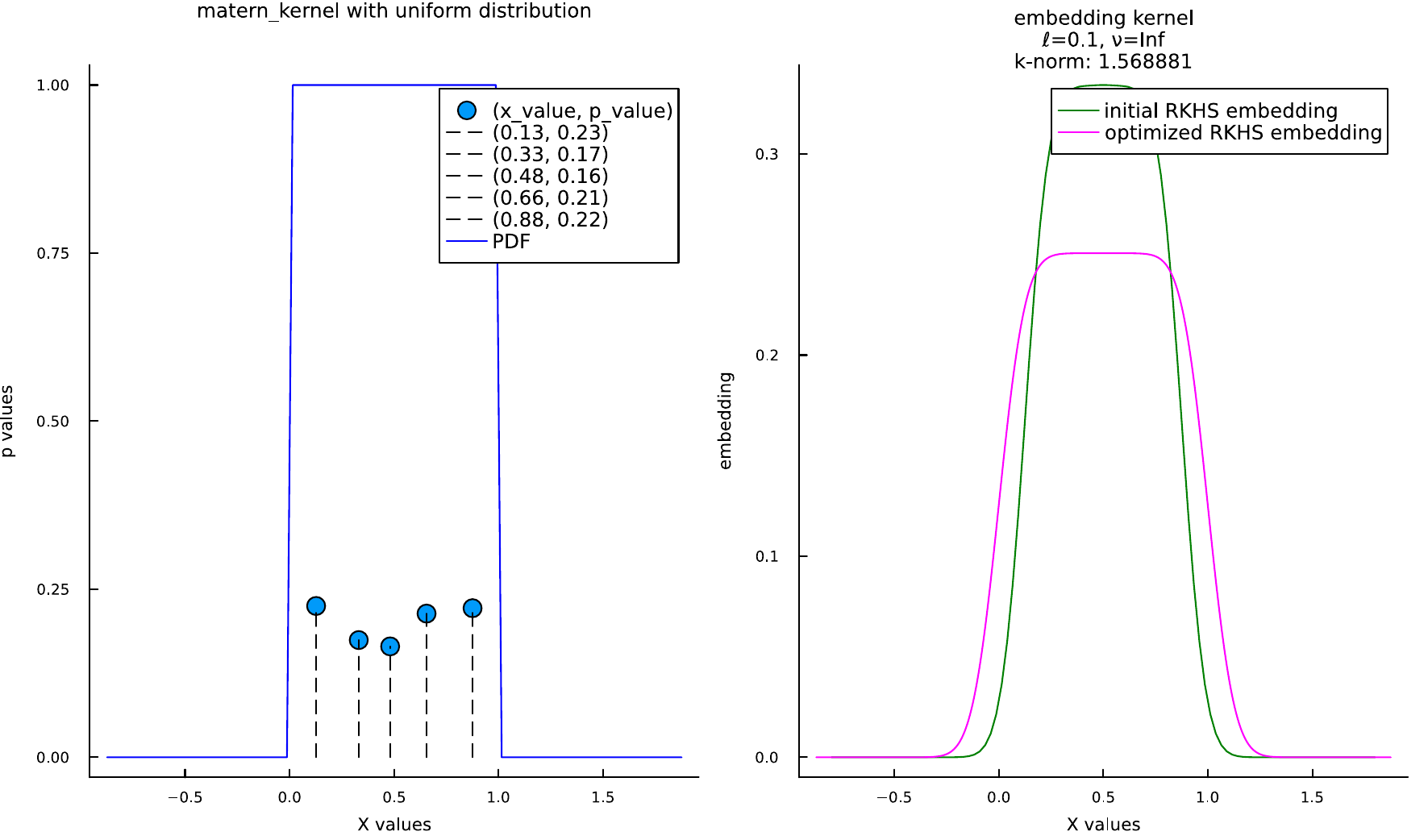}\label{fig:b21}}

  \subfloat[][$\ell=0.5$, $\nu=\nicefrac12$]
  {\includegraphics[width=0.49\textwidth]{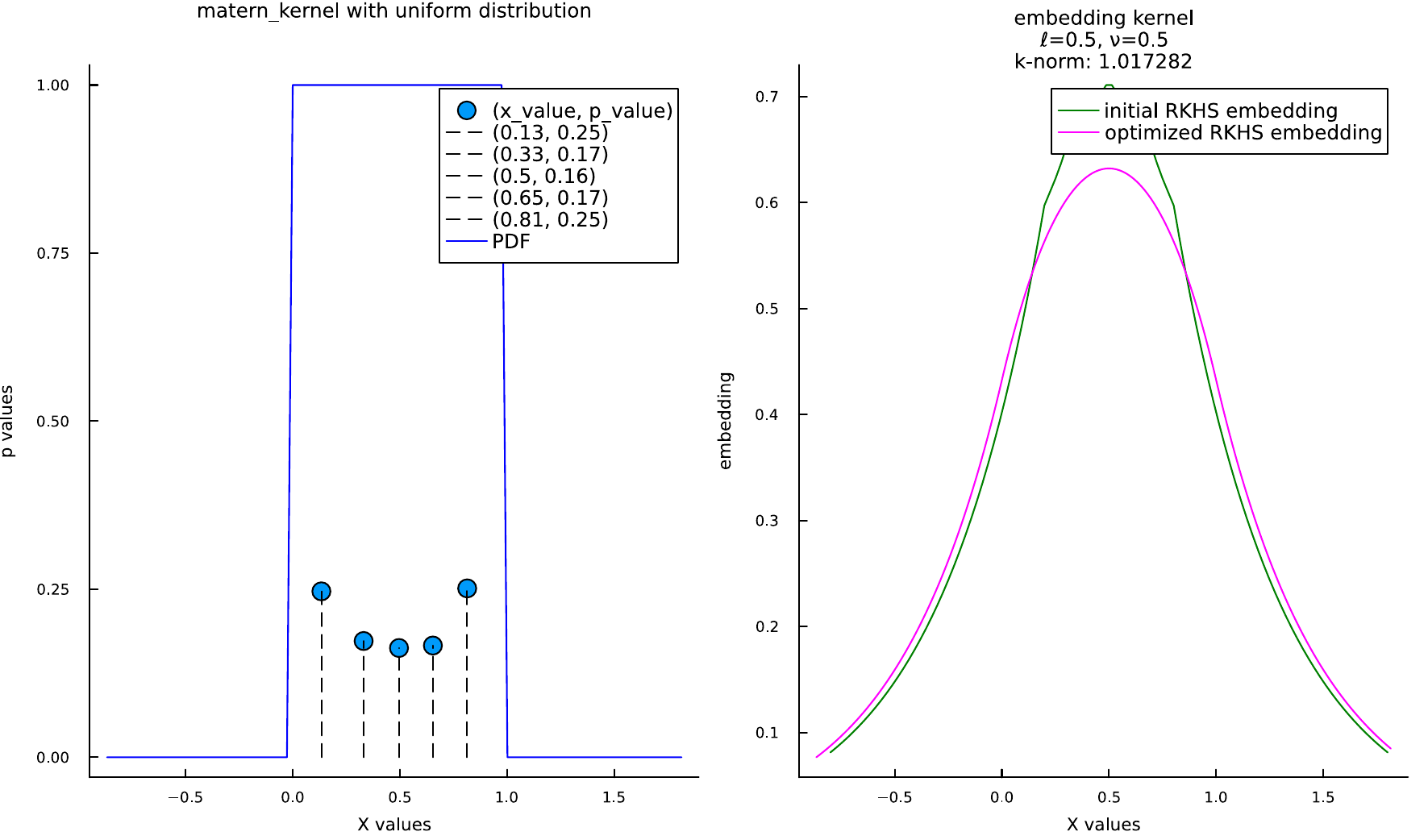}\label{fig:a22}}
  \hfill
  \subfloat[][$\ell=0.5$, $\nu= \infty$]
  {\includegraphics[width=0.49\textwidth]{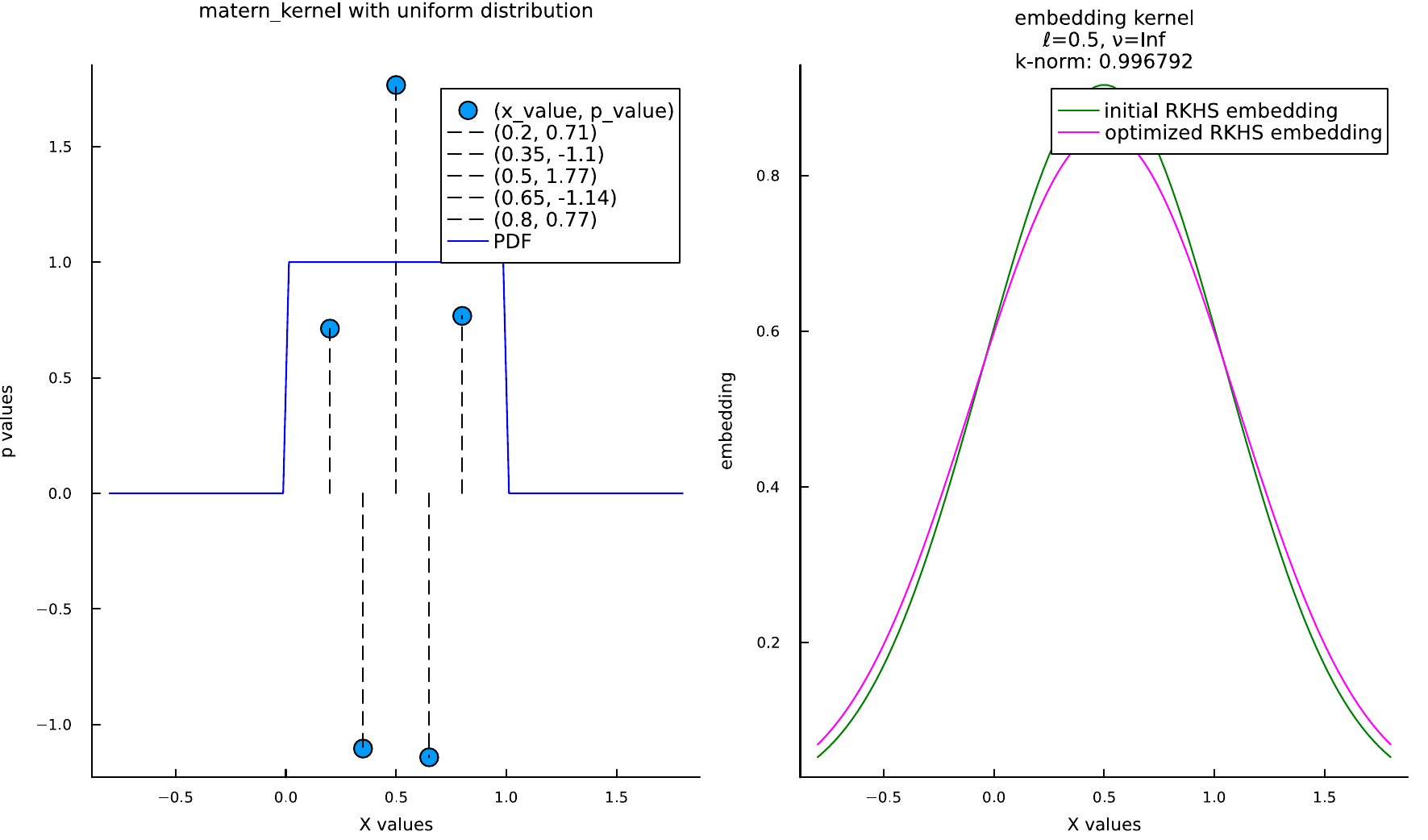}\label{fig:b22}}
  
  \caption{\label{fig:ab2}Approximations of the uniform distribution, cf.\ also Table~\ref{tab:uniform_distribution_results}}
\end{figure}





\begin{table}[H] 
    \centering
    \begin{tabular}{ccc}
        \toprule
        bandwidth~$\ell$ & smoothness~$\nu$ & $\MMD$~distance \\
        \midrule
        0.1 & 0.5 & 1.18629\\
        0.1 & 2.5 & 0.97844\\
        0.1 & $\infty$ & 1.56888\\
        0.5 & 0.5 & 1.01728\\
        0.5 & 2.5 & 0.69118\\
        0.5 & $\infty$ & 0.99679\\
        \bottomrule
    \end{tabular}
    \caption{\label{tab:uniform_distribution_results}Approximation distances obtained for the uniform distribution, cf.\ Figure~\ref{fig:ab2}}
\end{table}

\begin{figure}[]
  \centering
  \subfloat[][bandwidth $\ell=0.1$, smoothness $\nu= \nicefrac12$]
  {\includegraphics[width=0.49\textwidth]{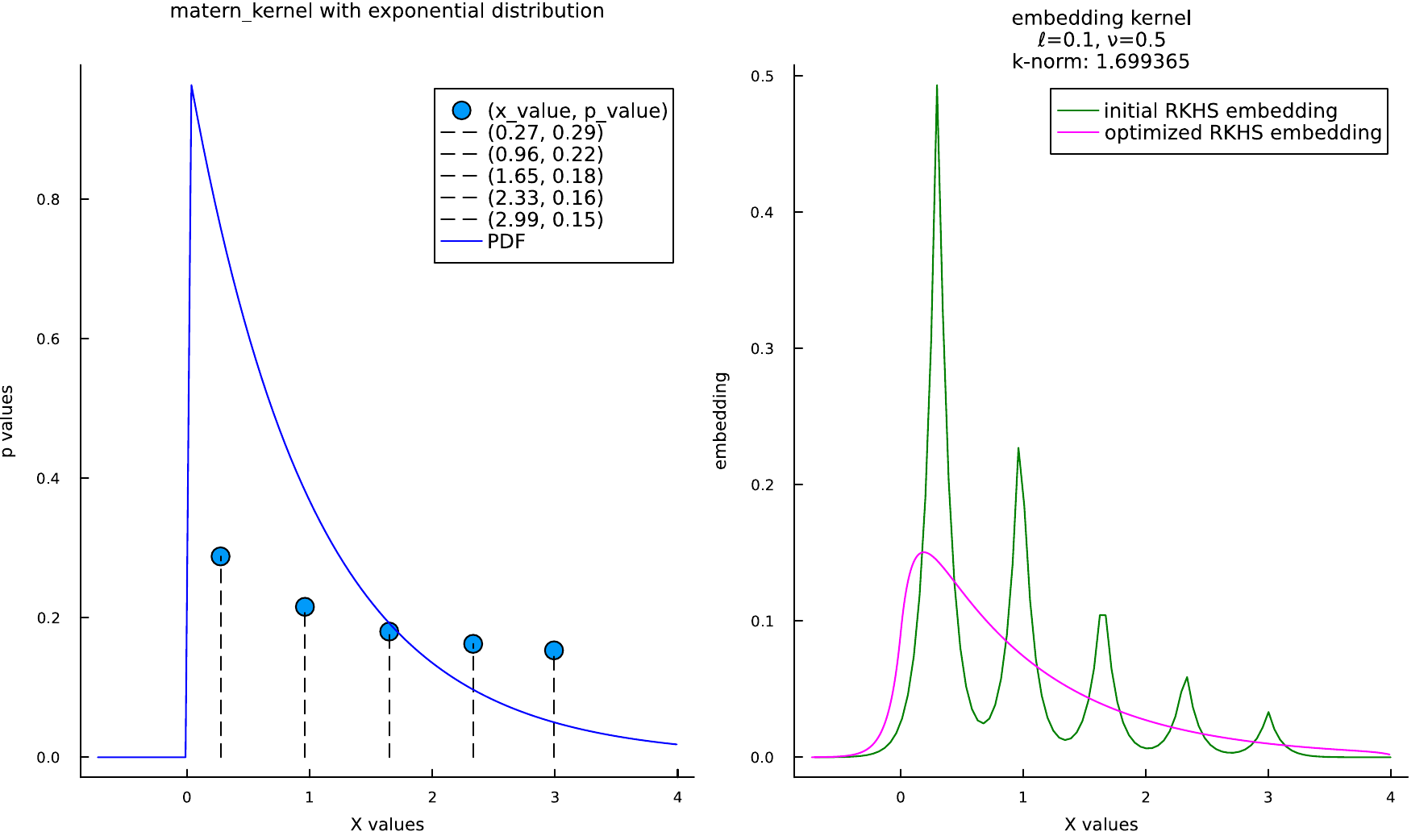}\label{fig:a31}}
  \hfill
  \subfloat[][$\ell=0.1$, $\nu=\infty$]
  {\includegraphics[width=0.49\textwidth]{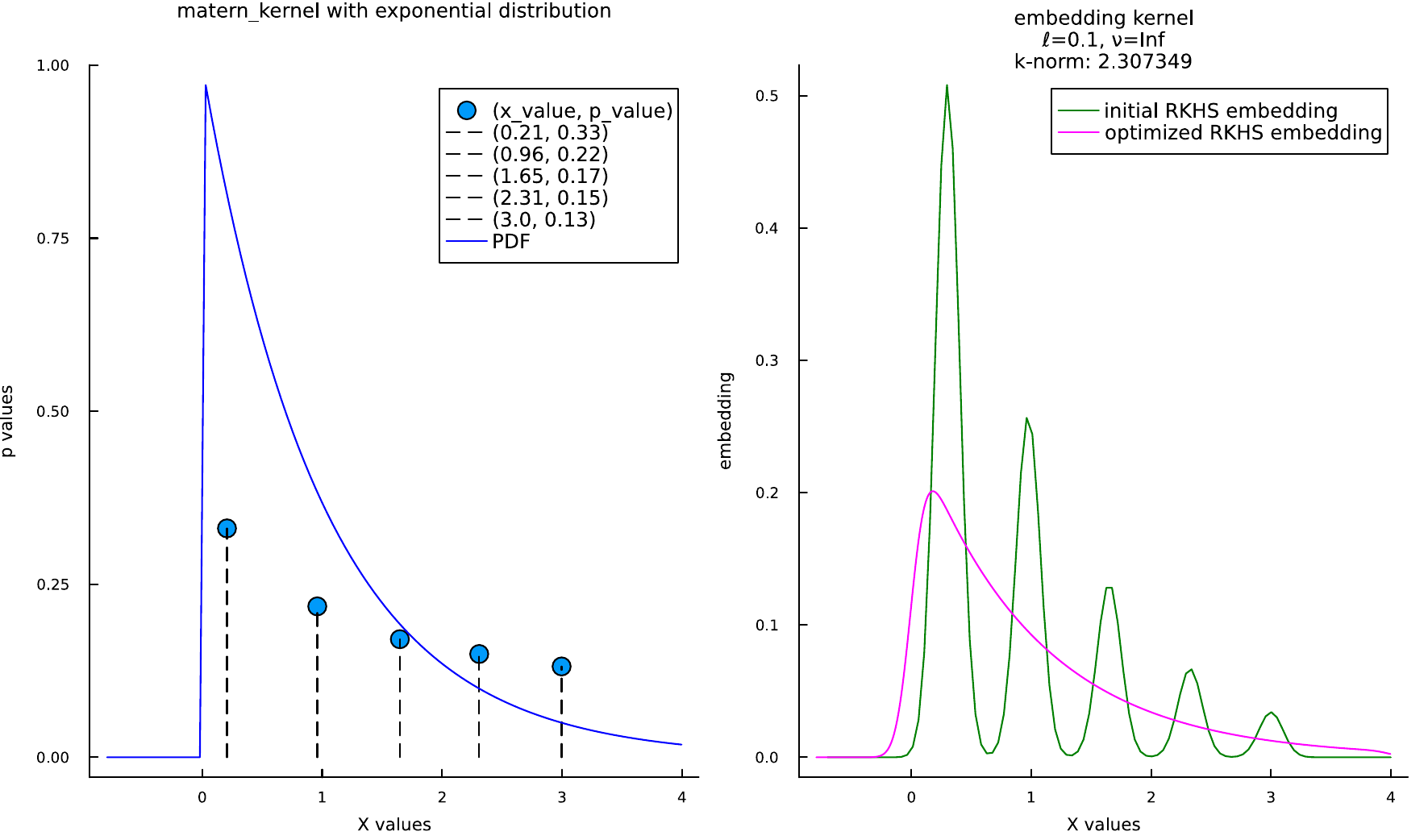}\label{fig:b31}}\\
  \subfloat[][bandwidth $\ell=0.1$, smoothness $\nu= \nicefrac12$]
  {\includegraphics[width=0.49\textwidth]{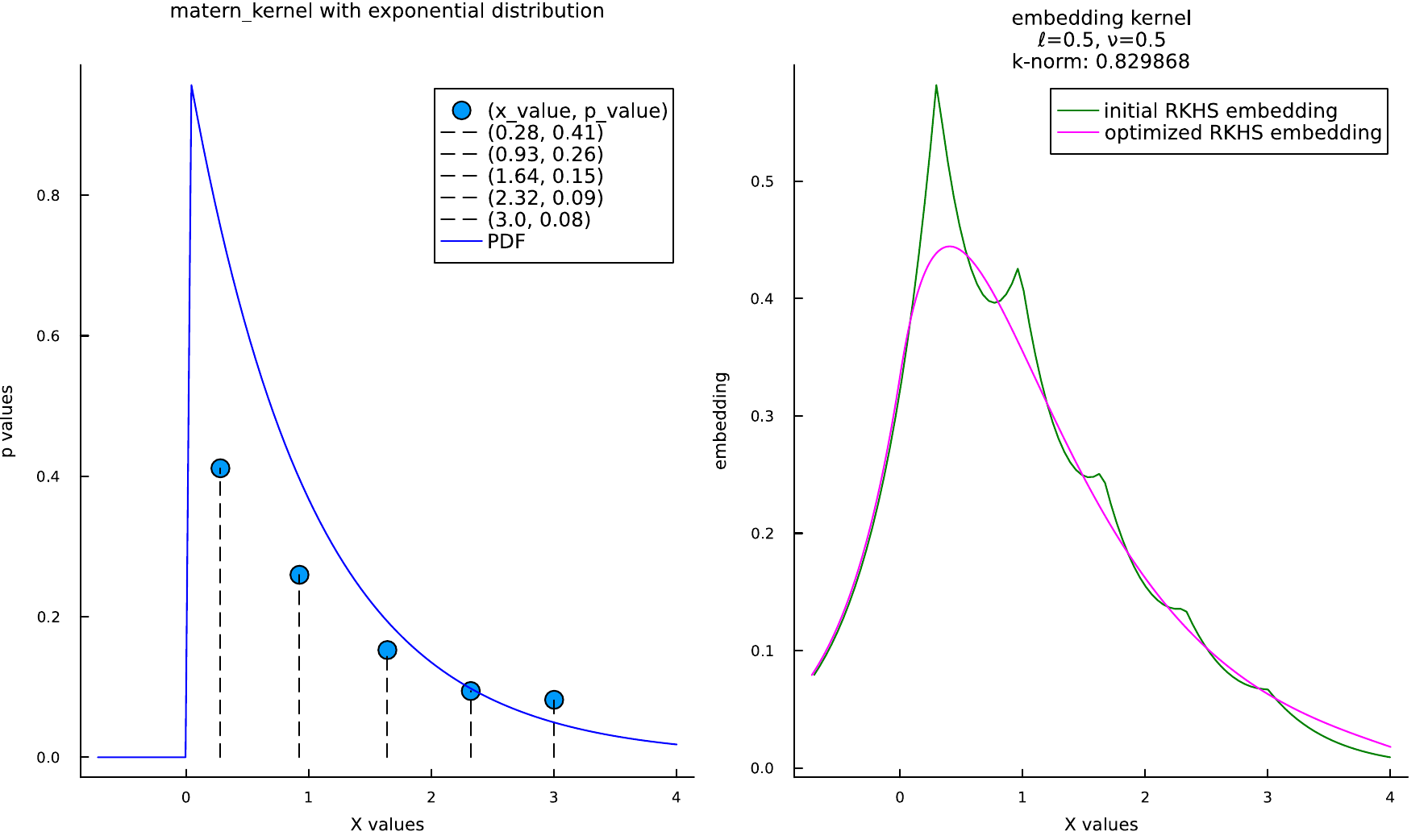}\label{fig:a32}}
  \subfloat[][$\ell=0.5$, $\nu=\infty$]
  {\includegraphics[width=0.49\textwidth]{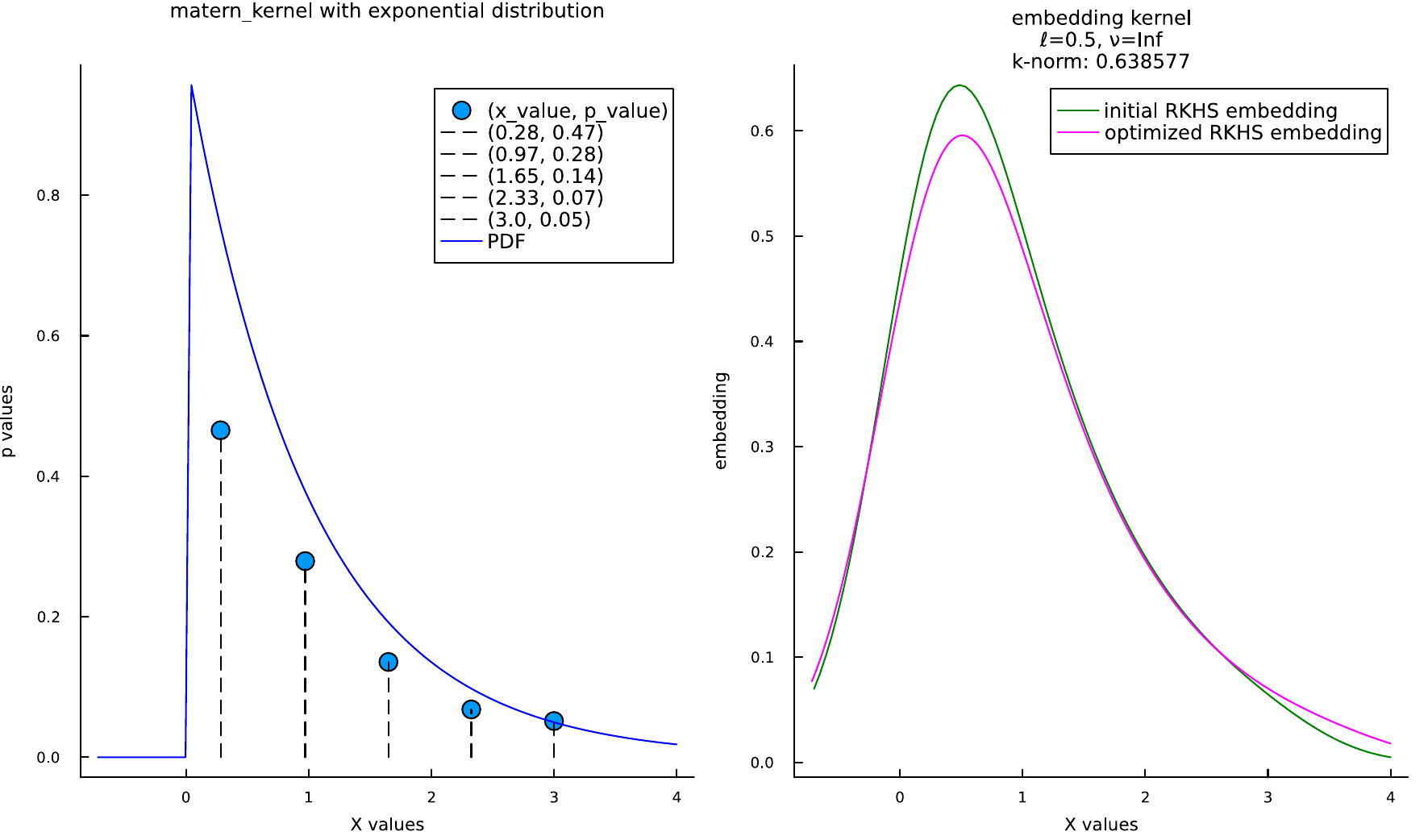}\label{fig:b2}}
  \caption{\label{fig:ab3}Quantization of the exponential distribution, cf.\ Table~\ref{tab:exponential_distribution_results}}
\end{figure}





\begin{table}[H] 
    \centering
    \begin{tabular}{ccc}
        \toprule
        bandwidth~$\ell$ & smoothness~$\nu$ & $\MMD$ distance  \\
        \midrule
        0.1 & 0.5 & 1.69937\\
        0.1 & 2.5 & 1.43278\\
        0.1 & $\infty$ & 2.30735\\
        0.5 & 0.5 & 0.82987\\
        0.5 & 2.5 & 0.42443\\
        0.5 & $\infty$ & 0.63858\\
        \bottomrule
    \end{tabular}
    \caption{\label{tab:exponential_distribution_results}Quantization quality of the exponential distribution, cf.\ Figure~\ref{fig:ab3}}
\end{table}

In conclusion, the results confirm the effectiveness and efficiency of our optimization approach across various distributions, demonstrating accurate solutions and robust performance. The stochastic gradient descent algorithm proves to be a versatile and powerful tool for kernel optimization in practical applications.

\section{Summary}\label{sec:1141}
The paper explores quantization in the context of maximum mean discrepancy distance for measures, with a particular focus on probability measures.
In quantization, a given measure is approximated by a simple measure that consists of many support points and corresponding weights. This approach makes the original measure tractable for numerical computations.

Typically, the approximating measure is not itself a probability measure.
To address this, the paper provides a method for determining the optimal weights, ensuring that the approximating measure is a probability measure.
The task of computing the optimal locations, however, presents a non-linear optimization problem. The paper introduces a reformulation on a product space, which allows for the application of stochastic gradient methods.

Future work will involve numerical experiments in higher dimensions, as well as the quantization of stochastic processes. This will link to the theory of nested optimal transport, as elaborated in \citet{PflugPichlerBuch}, for example.

\bibliographystyle{abbrvnat}
\bibliography{~/Dropbox/Literatur/LiteraturAlois,LiteratureMehraban}

\end{document}